\documentclass[pdflatex,sn-mathphys-num]{sn-jnl}% Math and Physical Sciences Numbered Reference Style
\usepackage{graphicx}%
\usepackage{multirow}%
\usepackage{amsmath,amssymb,amsfonts}%
\usepackage{amsthm}%
\usepackage{mathrsfs}%
\usepackage[title]{appendix}%
\usepackage{xcolor}%
\usepackage{textcomp}%
\usepackage{manyfoot}%
\usepackage{booktabs}%
\usepackage{algorithm}%
\usepackage{algorithmicx}%
\usepackage{algpseudocode}%
\usepackage{listings}%
\usepackage{epstopdf}
\usepackage{cases}
\usepackage{subcaption}
\usepackage{setspace}
\theoremstyle{thmstyleone}%
\newtheorem{theorem}{Theorem}%  meant for continuous numbers
\newtheorem{lemma}{Lemma}
\newtheorem{assumption}{Assumption}

\theoremstyle{thmstyletwo}%
\newtheorem{remark}{Remark}%

\theoremstyle{thmstylethree}%
\newtheorem{definition}{Definition}%

\begin{document}

\title[An accelerated proximal PRS-SQP algorithm with dual ascent-descent procedures for smooth composite optimization]{An accelerated proximal PRS-SQP algorithm with dual ascent-descent procedures for smooth composite optimization}


\author[1,2]{\fnm{Jiachen} \sur{Jin}}\email{jinjiachen@nudt.edu.cn}

\author[1]{\fnm{Guodong} \sur{Ma}}\email{mgd2006@163.com}

\author*[1]{\fnm{Jinbao} \sur{Jian}}\email{jianjb@gxu.edu.cn}

\affil[1]{\orgdiv{School of Mathematical Sciences, Center for Applied Mathematics of Guangxi}, \orgname{Guangxi Minzu University}, \city{Nanning}, \postcode{530006}, \state{Guangxi}, \country{China}}

\affil[2]{\orgdiv{College of Science}, \orgname{National University of Defense Technology}, \city{Changsha}, \postcode{410073}, \state{Hunan}, \country{China}}


\abstract{
	Conventional wisdom in composite optimization suggests augmented Lagrangian dual ascent (ALDA) in Peaceman-Rachford splitting (PRS) methods for dual feasibility. However, ALDA may fail when the primal iterate is a local minimum, a stationary point, or a coordinatewise solution of the highly nonconvex augmented Lagrangian function. Splitting sequential quadratic programming (SQP) methods utilize augmented Lagrangian dual descent (ALDD) to directly minimize the primal residual, circumventing the limitations of ALDA and achieving faster convergence in smooth optimization. This paper aims to present a fairly accessible generalization of two contrasting dual updates, ALDA and ALDD, for smooth composite optimization. A key feature of our PRS-SQP algorithm is its dual ascent-descent procedure, which provides a free direction rule for the dual updates and a new insight to explain the counterintuitive convergence behavior. Furthermore, we incorporate a hybrid acceleration technique that combines inertial extrapolation and back substitution to improve convergence. Theoretically, we establish the feasibility for a wider range of acceleration factors than previously known and derive convergence rates within the Kurdyka-{\L}ojasiewicz framework. Numerical experiments validate the effectiveness and stability of the proposed method in various dual-update scenarios.}




\maketitle

\section{Introduction}\label{sec1}
This paper considers the smooth composite optimization problem:
\begin{equation}\label{op}
	\min_{x}~f(x)+g(Ax),
\end{equation}
where $f: \mathbb{R}^{n_1} \to \mathbb{R}$ and $g: \mathbb{R}^{n_2} \to \mathbb{R}$ are smooth (not necessarily convex) functions, and $A: \mathbb{R}^{n_1} \to \mathbb{R}^{n_2}$ is a linear mapping. 
This model arises in machine learning and data science, where $f$ is a loss function (e.g., sigmoid loss \cite{Tibshirani96} or logistic loss \cite{RWL10}) and $g$ is a regularization term (e.g., Tikhonov regularization \cite{tikhonov1977solutions} or Huber function \cite{Huber1964}). 
Although the component functions are smooth, the composite structure presents a fundamental challenge to second-order methods. One example is pose graph optimization (PGO) for robotic simultaneous localization and mapping. In PGO, computing Hessians requires global pose information, and the complexity of rotations and translations increases as the graph grows, causing Newton-type methods to exhibit polynomial growth in computational complexity \cite{hu2018adaptive,chen2025non,chen2026parallelizable}. 
	
The Hessian of the objective function in (1) includes the term $A^\top\nabla^2g(Ax)A$, which couples the (often large-scale) data matrix $A$ with the inner curvature of $g$. This coupling makes it difficult to construct Hessian or quasi-Newton approximations that are simultaneously accurate, computationally tractable, and capable of exploiting potential sparsity. 	
A proven strategy for achieving efficiency is to exploit the separable structure inherent in \eqref{op} via operator splitting \cite{boyd11}. Introducing an auxiliary variable $y=Ax$ transforms problem \eqref{op} into a separable objective with linear constrains:
\begin{equation}\label{P}
	\min_{x, y}\{f(x)+g(y),\ {\rm s.t.}  \ A x=y\}.
\end{equation}
This reformulation yields two computational advantages. First, it decouples the composite structure, enabling $f(x)$ and $g(y)$ to be processed independently based on their specific properties. Second, it simplifies Hessian approximation, replacing the complex composite Hessian $\nabla^2(f+g\circ A)$ with independent approximations of $\nabla^2f$ and $\nabla^2g$.
Based on this reformulated foundation \eqref{P}, the sequential quadratic programming (SQP) method \cite{Wilson1963,FLT02,LMD16,AKR23,COR24} emerges as a natural and powerful tool. At each iteration $(x_k,y_k)$, the SQP method aims to solve the quadratic programming (QP) subproblem:
\begin{equation}\label{CQP}
	\min_{x,y}\{\nabla f(x_k)^{\top}(x-x_k)+\nabla g(y_k)^{\top}(y-y_k)+\frac{1}{2}\|(x-x_k, y-y_k)\|_{H_k}^2,~ 
	{\rm s.t.}~A x=y\},
\end{equation}
where $H_k$ is a symmetric approximation of the Hessian matrix for the Lagrangian function with respect to $(x, y)$, typically required to be positive definite in the null space of constraint Jacobian. 
However, solving the problem \eqref{CQP} remains computationally expensive, particularly for large-scale applications, since it is of the same size as \eqref{P}. To address this, Jian et al. \cite{JZY20,JLY21,JZY22,JMX24,JCT23} proposed a class of splitting SQP methods that approximates $\nabla^2 f(x_k)$ and $\nabla^2 g(y_k)$ with symmetric matrices $H_k^x$ and $H_k^y$ respectively. This enables problem \eqref{CQP} to be solved using effective splitting algorithms such as the alternating direction method of multipliers (ADMM) \cite{GM75} or the Peaceman-Rachford splitting (PRS) method \cite{PR55}, thereby reducing the computational cost. 

A defining feature of the splitting SQP method is its dual update: the dual variable $\lambda$ follows the steepest descent direction $-\nabla_\lambda L_{\beta}(x_{k+1},y_{k+1},\cdot)$, which is fundamentally different from classical primal-dual splitting methods like ADMM and PRS. 
For problem \eqref{P}, the PRS method employs an augmented Lagrangian dual ascent (ALDA) to update the dual variable, which essentially maximizes the concave augmented Lagrangian dual function $d(\lambda):=\min_{x,y} L_{\beta}(x,y,\lambda)$ via the gradient of $-d$. However, ALDA can struggle if the gap between $L_{\beta}(x_{k+1},y_{k+1},\lambda_k)$ and $d(\lambda_k)$ is large or fails to remain uniformly bounded over iterations \cite{SS24}. Counterintuitively, splitting SQP methods achieve faster convergence through augmented Lagrangian dual descent (ALDD), which directly targets primal feasibility by reducing the residual $\|Ax_{k+1}-y_{k+1}\|=\|\nabla_\lambda L_{\beta}(x_{k+1},y_{k+1},\cdot)\|$, bypassing the limitations of ALDA.
This motivates our study, which seeks to address two key questions:
\begin{itemize}
	\item Why does such an ALDD scheme, which seemingly undermines dual feasibility, preserve convergence?
	
	\item Is it possible to reconcile the two fundamentally different approaches, ALDD and ALDA, into one procedure?
\end{itemize}
Unlike existing works \cite{JZY20,JLY21,JZY22,JMX24,JCT23,SS24} that restrict dual updates to a given direction, this paper introduces a novel PRS-SQP algorithm with dual ascent-descent procedures that is free to set the dual update direction. 
By distilling and synthesizing two contrasting dual updates---ALDA and ALDD---into one procedure, we enable a comprehensive analysis of how the direction of the dual update affects the convergence of splitting SQP methods.
The following subsections review the literature on PRS methods and ALDD methods.

\subsection{Peaceman-Rachford splitting methods}
The PRS method is pivotal in splitting methods that exploit the separable structure of problem \eqref{P}. It can be seen as a symmetric variant of ADMM, with the main difference that PRS updates the dual variable $\lambda$ twice per iteration by ALDA. The PRS is not as ``robust'' as the ADMM in the sense that the former requires more restrictive conditions to ensure convergence. However, when both converge, PRS is often more efficient than ADMM \cite{lions1979splitting,GKM03}.
The iterative framework of PRS for solving problem \eqref{P} is as follows:
\begin{subequations}\label{PR}
	\begin{numcases}{}
		x_{k+1}=\arg\min_x \{L_{\beta}(x,y_k,\lambda_k)\},\label{PRX}\\
		\lambda_{k+\frac{1}{2}}=\lambda_{k}-\beta (Ax_{k+1}-y_k),\label{PRL1}\\
		y_{k+1}=\arg\min_y \{L_{\beta}(x_{k+1},y,\lambda_{k+\frac{1}{2}})\},\label{PRY}\\
		\lambda_{k+1}=\lambda_{k+\frac{1}{2}}-\beta (Ax_{k+1}-y_{k+1}),\label{PRL2}
	\end{numcases}
\end{subequations}
where the augmented Lagrangian function (ALF) is given by:
\begin{equation}\label{alf}
	L_{\beta}(x,y,\lambda)=f(x)+g(y)-{\lambda}^{\top}(Ax-y)+\frac{\beta}{2}\|Ax-y\|^2,
\end{equation} 
where $\beta>0$ is a penalty parameter and $\lambda$ is the dual variable (Lagrange multiplier). 
PRS methods minimize the primal variable $(x, y)$ via descent and maximize the dual variable $\lambda$ via ascent.
However, even in the convex setting for \eqref{P}, He et al. \cite{HLW14} noted that the sequence generated by \eqref{PR} is not guaranteed to converge. To address this, they proposed scaling the updates of $\lambda$ in \eqref{PRL1} and \eqref{PRL2} by a factor $\alpha\in (0,1)$ to enforce strict contractiveness.
Gu et al. \cite{GJH15} extended this approach by introducing two different relaxation factors and incorporating proximal terms in the subproblems \eqref{PRX} and \eqref{PRY}. The addition of a suitable proximal term not only ensures the uniqueness of the solution, but also provides a quantifiable descent in the ALF, which facilitates convergence analysis. The resulting proximal PRS iteration is given by
\[
\begin{cases}{}
	x_{k+1}=\arg\min\limits_{x} \{L_{\beta}(x,y_k,\lambda_k)+\frac{1}{2}\|x-x_k\|_S^2\},\\
	\lambda_{k+\frac{1}{2}}=\lambda_{k}-r\beta (Ax_{k+1}-y_k),\\
	y_{k+1}=\arg\min\limits_{y} \{L_{\beta}(x_{k+1},y,\lambda_{k+\frac{1}{2}})+\frac{1}{2}\|y-y_k\|_T^2\},\\
	\lambda_{k+1}=\lambda_{k+\frac{1}{2}}-s\beta (Ax_{k+1}-y_{k+1}),
\end{cases}
\]
where $r\in [0,1)$ and $s\in (0,\frac{1-\alpha+\sqrt{(1+\alpha)^2+4(1-\alpha)^2}}{2})$ are dual stepsizes, and $S$ and $T$ are symmetric and possibly indefinite matrices defining $\|x\|_{S}^2={x^{\top}Sx}$ and $\|x\|_{T}^2={x^{\top}Tx}$. 

To further improve PRS methods, some acceleration techniques have been explored, broadly categorized into two approaches. The first integrates back substitution \cite{HTY12}, which combines the current and previous iterates to stabilize convergence \cite{DL19}. The second employs inertial extrapolation \cite{AA01}, where the algorithm extrapolates the current iterate in the direction of the last movement before applying PRS, thereby accelerating convergence \cite{DLL17,DL20}. These studies, however, focus on convex settings.
While nonconvex PRS methods have been extensively analyzed \cite{CDD16,LLP17,GHLILW22,LJH23}, their accelerated variants incorporating back substitution or inertial extrapolation for primal variables remain underexplored, with only \cite{WSL23} to our knowledge investigating this problem.

\subsection{Augmented Lagrangian dual descent methods}
In recent years, an ALDD scheme has emerged as an alternative to the ALDA \cite{JZY20,JLY21,JZY22,JCT23,JMX24,SS24}. Unlike ADMM or PRS, this novel update moves iterations in parallel but opposite directions, representing a fundamentally different approach. 
Jian et al. \cite{JZY20,JLY21,JZY22,JCT23,JMX24} pioneered the use of ALDD in the following splitting SQP method:
	\[\left\{\begin{array}{l}
		\tilde{x}_{k+1}=\arg\min\limits_{x} \left\{\nabla_x L_{\beta}(x_k,y_k,\lambda_k)^\top (x-x_k)+\frac{1}{2}\|x-x_k\|^2_{H_x^k+\beta A^\top A}\right\},\\
		\tilde{y}_{k+1}=\arg\min\limits_{y} \left\{\nabla_y L_{\beta}(x_k,y_k,\lambda_k)^\top (y-y_k)+\frac{1}{2}\|y-y_k\|^2_{H_y^k+\beta I_{n_2}}\right\},\\
		d^x_k=\tilde{x}_{k+1}-x_k,~d^y_k=\tilde{y}_{k+1}-y_k,~d^u_k=(d^x_k,d^y_k),~u_k=(x_k,y_k),\\
		L_{\beta}(u_k+t_kd^u_k,\lambda_k)\leq L_{\beta}(u_k,\lambda_k)-\rho t_k (\|d^x_k\|^2_{H_x^k+\beta A^\top A}+\|d^y_k\|^2_{H_y^k+\beta I_{n_2}}),\\
		x_{k+1}=x_k+t_k d^x_k,~y_{k+1}=y_k+t_k d^y_k,~\lambda_{k+1}=\lambda_{k}+\xi(Ax_{k+1}-y_{k+1}),
	\end{array}\right.\]
where $\xi>0$, $\rho\in (0, 1)$, and stepsize $t_k$ is determined via the line search. 
Assuming the dual variable is bounded, they proved that the method converges to an $\epsilon$-stationary solution with a complexity of $\mathcal{O} (\epsilon^{-2})$ \cite{JZY20,JLY21}. Inspired by the PRS method \eqref{PR}, they further introduced a hybrid dual update that combines both ALDD and ALDA, correcting the multiplier $\lambda$ twice per iteration \cite{JZY22}. The dual stepsizes in \cite{JZY22} have to satisfy $r+s<0$, so it is still an ALDD method. This framework was later extended to solve general linearly constrained optimization problems \cite{JCT23,JMX24}. In these works, the ALDD was interpreted as a mechanism to ensure the monotonicity of the iteration sequence.
Then a seminal work by Sun and Sun \cite{SS24} considered a scaled stepsize for the ALDD and incorporated a regularization term in $\lambda$ with $\omega>0$ into the ALF \eqref{alf}.
Upon obtaining $u_{k+1}:=(x_{k+1},y_{k+1})$, $\lambda_{k+1}$ is updated via a proximal gradient step,
\[\lambda_{k+1}=\arg\min_\lambda \{L_{\beta}(u_{k+1},\lambda)+\frac{\omega}{2\beta}\|\lambda\|^2+\frac{\tau\omega}{2\beta}\|\lambda-\lambda_k\|^2\}
=\frac{\tau}{1+\tau}(\lambda_k+\frac{\beta}{\tau\omega}(Ax_{k+1}-y_{k+1})),\]
where $\tau>0$. They established that the scaled dual descent ADMM achieves an $\epsilon$-stationary solution with a complexity of $\mathcal{O} (\epsilon^{-4})$. 
To date, only a handful of studies have examined dual descent methods, leaving significant gaps in understanding. This motivates our in-depth investigation of this promising yet underexplored direction.

\subsection{Contributions}
The main contributions are summarized as follows:

The algorithmic framework in this paper, called HAP-PRS-SQP (Algorithm \ref{algo1}), generalizes dual updates by allowing non-restrictive dual stepsizes $r$ and $s$ to satisfy $r+s\ne 0$, accommodating both ascent ($r>0,s>0$) and descent ($r<0,s<0$) directions. This flexibility seamlessly integrates ALDA and ALDD into one procedure.
Our analysis shows that the direction of the dual update does not substantially change the monotonicity of the ALF-based merit function (Lemma \ref{lem3}), thus ensuring convergence while accommodating both ascent and descent updates.

The proposed hybrid acceleration strategy unifies and extends back substitution \cite{HTY12} and inertial extrapolation \cite{AA01}. Unlike these techniques, our strategy is embedded between two steps, acting as a substitution for the forward step and an extrapolation for the backward step. 
Theoretically, we prove the convergence of HAP-PRS-SQP with an upper bound on the acceleration factor of $1/\rho-1,\rho\in (0,0.5)$, which exceeds the previously known result of 1 in \cite{DLL17}. Furthermore, within the Kurdyka-{\L}ojasiewicz (KL) framework, we establish the convergence rate of HAP-PRS-SQP.

Through illustrative examples, we demonstrate that ALDA enhances feasibility while ALDD accelerates convergence by reducing primal residuals. Numerical results on  $L_2$ regularized binary classification problem and smooth LASSO model validate that HAP-PRS-SQP outperforms the classical SQP and existing splitting SQP methods.

The rest of this paper is organized as follows. Section \ref{sec2} presents some notations and preliminary results that will be used throughout this paper. Section \ref{sec3} describes the algorithm in detail. The convergence analysis can be found in Section \ref{sec4}, followed by numerical result in Section \ref{sec5}. Concluding remarks are given in Section \ref{sec6}.

\section{Preliminaries}\label{sec2}
\textbf{Notation.} we denote $\mathbb{R}$ as the real number set and $\|\cdot\|$ as for the Euclidean norm on the $n$-dimensional Euclidean space $\mathbb{R}^{n}$. $I_n\in\mathbb{R}^{n\times n}$ is an identity matrix of size $n$. $S\succ0(\succeq0)$ represents that the matrix $S$ is (semi)positive definite.
For any $x,y\in\mathbb{R}^{n},~S\in\mathbb{R}^{n\times n}$, $\|x\|_{S}^2={x^{\top}Sx}$, $\langle x,y\rangle=x^\top y$.
The minimum and maximum eigenvalues of symmetric matrix $S$ are denoted by $\lambda_{\min}(S)$ and $\lambda_{\max}(S)$, respectively. For any vector $x$, $\lambda_{\min}(S)\|x\|^{2}\leq\|x\|^{2}_{S}\leq\lambda_{\max}(S)\|x\|^{2}$.
The distance from any point $x\in \mathbb{R}^n$ to any subset $\mathcal{Q}\subseteq \mathbb{R}^n$ is $d(x,\mathcal{Q}):= \inf_{y\in \mathcal{Q}}\|y-x\|$, and if $\mathcal{Q}=\emptyset$, let $d(x,\mathcal{Q})=+\infty$.

\begin{lemma} {\rm \cite{Nesterov2018}}
For a continuous differentiable function $h: \mathbb{R}^n\to\mathbb{R}$, its gradient is Lipschitz continuous with the modulus $L>0$, then,
$$|h(y)-h(x)-\langle\nabla h(x),y-x\rangle|\leq\frac{L}{2}\|y-x\|^{2},~\forall~x,y\in\mathbb{R}^n.$$
\end{lemma}

The domain of a function $q$ is denoted by ${\rm dom}~q=\{x\in\mathbb{R}^n:q(x)<+\infty\}$. For a proper lower semicontinuous function $q:\ \mathbb {R}^n\to \mathbb {R}\cup \{+\infty\}$ and $-\infty<\sigma _{1}<\sigma _{2}<+\infty$, define
$[\sigma _{1}<q<\sigma _{2}]:=\{x\in\mathbb{R}^n:\sigma _{1}<q(x)<\sigma _{2}\}.$

\begin{definition}{\rm \cite{ABS13} (KL property)
For a proper lower semicontinuous function  $q:\mathbb {R}^n\to \mathbb {R}\cup \{+\infty\}$ and point $\bar{x}\in {\rm dom}~q$, if there exist a neighborhood $U$ of $\bar{x}$, a continuous concave function $\varphi:[0,\sigma )\to\mathbb{R}_{+}$ and $\sigma=\sigma(\bar{x}) \in(0,\infty]$ such that
(i) $\varphi(0) = 0$;
(ii) $\varphi$ is continuous differentiable function on $(0,\sigma )$;
(iii) $\varphi'(t)>0,~\forall~t\in (0,\sigma )$;
(iv) $\varphi'(q(x)-q(\bar{x}))d(0,\partial q(x))\geq1,~\forall~x\in U\cap[q(\bar{x})<q<q(\bar{x})+\sigma ]$,
then $q$ is said to have the KL property at $\bar{x}$ and $\varphi$ is the KL correlation function of $q$.
}\end{definition}

The KL property is a geometric condition relating to the graph of a function near its critical points. It holds for a broad and practically relevant class of functions, including semi-algebraic functions (e.g., polynomials and rational functions), subanalytic functions (e.g., exp, log, and tanh); and, more generally, tame functions definable in an o-minimal structure, which is a vast, stable class encompassing the above \cite{attouch2009convergence,attouch2013convergence}. We denote by $\Phi_\sigma$ the class of $\varphi$ which satisfies the properties (i)-(iii) above. If $q$ satisfies the KL property at each point of ${\rm dom}~q$, then $q$ is called a KL function.

\begin{lemma}\label{KLP}{\rm{\cite{BST14} (Uniform KL property)}} Assume that a proper lower semicontinuous function $q:\mathbb {R}^n\to \mathbb {R}\cup \{+\infty\}$ is constant on a compact set $\Omega$ and satisfies the KL property at each point of $\Omega$. Then, there exist $\epsilon>0,~\sigma>0$ and $\varphi\in\Phi_\sigma$ such that
$\varphi'(q(x)-q(\bar{x}))d(0,\partial q(x))\geq 1$ holds for any $ \bar{x}\in\Omega$  and $x\in\{x\in \mathbb {R}^n:d(x,\Omega)<\epsilon\}\cap[q(\bar{x})<q(x)<q(\bar{x})+\sigma]$.
\end{lemma}

Here $\partial q(x)$ is the limiting subdifferential of $q$. If $q$ is continuously differentiable, then $\partial q(x)=\nabla q(x)$. This paper use these results in the smooth case. Let $w=(x,y,\lambda)$, the ALF \eqref{alf} implies
\begin{equation}\label{F2}
\left\{\begin{array}{l}
\nabla _x L_{\beta}(w)=\nabla f(x)-A^{\top}(\lambda-\beta(A x-y)),\\
\nabla _y L_{\beta}(w)=\nabla g (y)+\lambda-\beta(A x-y),\\
\nabla _\lambda L_{\beta}(w)=-(Ax-y).
\end{array}\right.
\end{equation}
\begin{definition}{\rm A point $(x^*,y^*)$ is said to be a Karush-Kuhn-Tucker (KKT) point to problem \eqref{P} with multiplier $\lambda^*$, if
\[\nabla f(x^*)=A^\top \lambda^*,\ \nabla g(y^*)=-\lambda^*,\ Ax^*=y^*.\]
}\end{definition}
Obviously, $(x^*,y^*)$ is a KKT point for \eqref{P} with $\lambda^*$ if and only if $\nabla_w L_{\beta}(x^*,y^*,\lambda^*)=0$, i.e., $w^*=(x^*,y^*,\lambda^*)$ is a stationary point of $L_{\beta}(w)$. We denote the set of all stationary points of $L_{\beta}$ as crit $L_{\beta}$.
Since the constraint in \eqref{P} is affine linear, a KKT point is a necessary condition for a local optimal solution to \eqref{P}. Consequently, our focus in the next section is on designing a novel algorithm to compute a KKT point for problem \eqref{P}.

\section{Algorithm design}\label{sec3}
For the $k$-th iteration $(x_k,y_k)$ in \eqref{CQP}, we choose $H_k={\rm diag}(H_k^x,H_k^y)$, where $H_k^x$ and $H_k^y$ are symmetric approximations of $\nabla^2 f(x_k)$ and $\nabla^2 g(y_k)$, respectively. This transforms the problem into a two-block QP problem:
\[\min_{x, y}\{F_k^f(x)+G_k^g(y),\ {\rm s.t.}  \ A x=y\} ,\]
where
\[
F_k^f(x)=\nabla f(x_k)^{\top}(x-x_k)+\frac{1}{2}\|x-x_k\|_{H_k^x}^2,
G_k^g(y)=\nabla g(y_k)^{\top}(y-y_k)+\frac{1}{2}\|y-y_k\|_{H_k^y}^2.
\]
We adopt a proximal PRS method, which decomposes this problem into two small-scale QP subproblems for $x$ and $y$ that can be solved separately.

\textbf{$x$-subproblem:} For the current iterate $w_k=(x_k,y_k,\lambda_k)$, the $x$-subproblem, augmented with a proximal term, is formulated as:
\begin{align}
	\tilde{x}_{k+1}&=\arg\min_x\{F_k^f(x)-\lambda_{k}^{\top}(A x-y_{k})+\frac{\beta}{2}\|A x-y_{k}\|^2+ \frac{\ell}{2}\|x-x_k\|^2\},\nonumber\\
	&=\arg\min_x\{[\nabla f(x_{k})-A^{\top}(\lambda_{k}-\beta(A x_{k}-y_{k}))]^{\top}(x-x_{k})+\frac{1}{2}\|x-x_{k}\|_{\mathcal{H}_k^x}^{2}\},\label{xQP}
\end{align}
where $\mathcal{H}_k^x=H_k^x+\beta A^{\top}A+\ell I_{n_1}$ and $\ell>0$. If $\mathcal{H}_k^x\succ0$, the solution $\tilde{x}_{k+1}$ is a unique. From the KKT optimality conditions, we have
\begin{equation}\label{xoc}
\nabla f(x_{k})-A^{\top}(\lambda_{k}-\beta(A x_{k}-y_{k}))+\mathcal{H}_k^x(\tilde{x}_{k+1}-x_{k})=0.
\end{equation}
To accelerate convergence, we introduce a {\it hybrid acceleration step} inspired by inertial techniques \cite{AA01} and accelerated strategies \cite{GCH20}:
\begin{equation}\label{xines}
\bar{x}_{k+1}=\tilde{x}_{k+1}+\alpha(\tilde{x}_{k+1}-x_{k}),
\end{equation}
where $\alpha>-1$. This yields the search direction for $x$:
\begin{equation}\label{sd-x}
d_{k}^{x} = \bar{x}_{k+1}-x_{k}=(1+\alpha)(\tilde{x}_{k+1}-x_{k}).
\end{equation}
From \eqref{F2}, \eqref{xoc} and \eqref{sd-x}, it follows that
\begin{align}
\nabla_x L_{\beta}(w_k)^{\top}d_k^x=- \frac{1}{1+\alpha}\|d_k^x\|_{\mathcal{H}_k^x}^2
=-(1+\alpha)\|\tilde{x}_{k+1}-x_{k}\|_{\mathcal{H}_k^x}^2.
\label{xd}
\end{align}
Thus, for $\alpha>-1$, $L_{\beta}(x,y_k,\lambda_k)$ exhibits a descent property along $d_{k}^{x}$ at $x_k$. We use $L_{\beta}(x,y_k,\lambda_k)$ as the merit function and perform a line search along $d_{k}^{x}$ to generate next iterate $x_{k+1}$.

\textbf{First dual update}: Based on $(x_{k+1},y_k)$, the multiplier $\lambda_{k}$ is updated as:
\begin{equation}\label{lambda1}
\lambda_{k+\frac{1}{2}} = \lambda_{k}-r \beta(A x_{k+1}-y_{k}).
\end{equation}
where $r\in \mathbb{R}$.

\textbf{$y$-subproblem}: Similarly, for the iterate $(x_{k+1},y_k,\lambda_{k+\frac{1}{2}})$, the $y$-subproblem is
\begin{align}
\tilde{y}_{k+1}&=\arg\min_y\{G_k^g(y)-\lambda_{k+\frac{1}{2}}^{\top}(A x_{k+1}-y)+\frac{\beta}{2}\|A x_{k+1}-y\|^2+\frac{\sigma}{2}\|y-y_{k}\|^{2}\},\nonumber\\
&=\arg\min_y\{[\nabla g (y_{k})+\lambda_{k+\frac{1}{2}}-\beta(A x_{k+1}-y_{k})]^{\top}(y-y_{k})+\frac{1}{2}\|y-y_{k}\|_{\mathcal{H}_k^y}^{2}\},\label{yQP}
\end{align}
where $\label{yQP_H}\mathcal{H}_k^y=H_k^y+(\beta+\sigma) I_{n_2}$ and $\sigma>0$. If $\mathcal{H}_k^y\succ0$, the solution $\tilde{y}_{k+1}$ is unique, and the KKT conditions yield:
\begin{equation}\label{yoc}
\nabla g (y_{k})+\lambda_{k+\frac{1}{2}}-\beta(A x_{k+1}-y_{k})+\mathcal{H}_k^y(\tilde{y}_{k+1}-y_{k})=0.
\end{equation}
Applying the hybrid acceleration step:
\begin{equation}\label{yines}
\bar{y}_{k+1}=\tilde{y}_{k+1}+\alpha(\tilde{y}_{k+1}-y_{k}),
\end{equation}
we obtain the search direction for $y$:
\begin{equation}\label{sd-y}
d_{k}^{y} = \bar{y}_{k+1}-y_{k}=(1+\alpha)(\tilde{y}_{k+1}-y_{k}).
\end{equation}
From \eqref{F2}, \eqref{yines} and \eqref{sd-y}, it follows that:
\[\begin{aligned}
\nabla _y L_{\beta}(x_{k+1},y_k,\lambda_{k+\frac{1}{2}})^{\top}d_k^y
= - \frac{1}{1+\alpha}\|d_k^y\|_{\mathcal{H}_k^y}^2
=-(1+\alpha)\|\tilde{y}_{k+1}-y_{k}\|_{\mathcal{H}_k^y}^2.\label{yd}
\end{aligned}\]
Thus, \eqref{yQP} provides a descent direction $d_{k}^{y}$ for $y_k$ at $(x_{k+1},y_k,\lambda_{k+\frac{1}{2}})$.
We perform a line search along $d_{k}^{y}$ to compute the next iterate $y_{k+1}$.

\textbf{Second dual update:}
Finally, the multiplier is updated a second time:
\begin{equation}\label{lambda2}
\lambda_{k+1} = \lambda_{k+\frac{1}{2}}-s \beta(A x_{k+1}-y_{k+1}),
\end{equation}
where $s\in \mathbb{R}$.
The hybrid acceleration proximal Peaceman-Rachford splitting SQP (HAP-PRS-SQP) algorithm is summarized in Algorithm \ref{algo1}.
\begin{algorithm}[]
\caption{HAP-PRS-SQP algorithm}\label{algo1}

\noindent {\bf Step 0. (Initialization)} Given an initial point $w_0=(x_0,y_0,\lambda_0)$. Select parameters $\rho,\nu\in (0,1)$, $\alpha \in (-1, \frac{1}{\rho}-1)$, penalty parameter $\beta>0$, proximal parameters $\ell,\sigma>0$, multiplier stepsizes $r$ and $s$ satisfying $r+s\not=0$. Choose positive definite matrices $H_0^x \in \mathbb{R}^{n_1\times n_1}$ and $H_0^y\in \mathbb{R}^{n_2\times n_2}$. Compute $\mathcal{H}_0^x=H_0^x+\beta A^{\top}A+\ell I_{n_1}$ and $\mathcal{H}_0^y=H_0^y+(\beta +\sigma) I_{n_2}$. Set $k:= 0$.

\noindent {\bf Step 1. (Solve the $x$-subproblem)} (i) Solve the subproblem \eqref{xQP} to obtain $\tilde{x}_{k+1}$;

(ii) Compute the accelerated iterate $\bar{x}_{k+1}$ using \eqref{xines};

(iii) Calculate the search direction $d_{k}^{x} = \bar{x}_{k+1}-x_{k}$ and determine the stepsize $t_k^x=\nu^{i_k}$, where $i_k$ is the smallest nonnegative integer $i$ satisfying:
\begin{equation}\label{xlines}
L_{\beta}(x_{k}+\nu^{i_k}d_{k}^{x}, y_{k}, \lambda_{k}) \leq L_{\beta}(x_{k}, y_{k}, \lambda_{k})-\rho \nu^{i_k}\|d_{k}^{x}\|_{\mathcal{H}_k^x}^{2}.
\end{equation}

(iv) Update the iterate:
\begin{equation}\label{newx}
x_{k+1}=x_{k}+t_{k}^{x} d_{k}^{x}.
\end{equation}

\noindent {\bf Step 2. (First dual update)} Update the multiplier $\lambda_{k}$ using \eqref{lambda1} to obtain $\lambda_{k+\frac12}$.

\noindent {\bf Step 3. (Solve the $y$-subproblem)} (i) Solve the subproblem \eqref{yQP} to obtain $\tilde{y}_{k+1}$;

(ii) Compute the accelerated iterate $\bar{y}_{k+1}$ using \eqref{yines};

(iii) Calculate the search direction $d_{k}^{y} = \bar{y}_{k+1}-y_{k}$ and determine the stepsize $t_k^y=\nu^{i_k}$, where $i_k$ is the smallest nonnegative integer $i$ satisfying:
\begin{equation}\label{ylines}
L_{\beta}(x_{k+1}, y_{k}+\nu^{i_k} d_{k}^{y}, \lambda_{k+\frac{1}{2}}) \leq L_{\beta}(x_{k+1}, y_{k}, \lambda_{k+\frac{1}{2}})- \rho \nu^{i_k}\|d_{k}^{y}\|_{\mathcal{H}_k^y}^{2}.
\end{equation}

(iv) Update the iterate:
\begin{equation}\label{newy}
y_{k+1}=y_{k}+t_{k}^{y} d_{k}^{y}.
\end{equation}

\noindent {\bf Step 4. (Second dual update)} Update the multiplier $\lambda_{k+\frac{1}{2}}$ using \eqref{lambda2} to obtain $\lambda_{k+1}$. Generate the new iteration $w_{k+1}=(x_{k+1}, y_{k+1},\lambda_{k+1})$.

\noindent {\bf Step 5. (Termination check)} If $w_{k+1}-w_k=0$, stop. Otherwise, proceed to Step 6.

\noindent {\bf Step 6. (Update Hessian approximations)} Compute matrices $H_{k+1}^x$ and $H_{k+1}^y$ as reasonable approximations of $\nabla ^2 f (x_{k+1})$ and $\nabla ^2 g (y_{k+1})$, ensuring: $\mathcal{H}_{k+1}^x:=H_{k+1}^x+\beta A^{\top}A+\ell I_{n_1}\succ 0$ and $\mathcal{H}_{k+1}^y:=H_{k+1}^y+(\beta +\sigma) I_{n_2}\succ 0$. Set $k:= k+1$ and go to Step 1.
\end{algorithm}

\begin{remark}\label{rem2}{\rm We give some remarks on Algorithm \ref{algo1}.

(i) The main computation in the HAP-PRS-SQP algorithm is the solving of subproblems \eqref{xQP} and \eqref{yQP}. These are standard QP problems and can be efficiently solved using existing solvers such as OOQP\footnote{\href{https://www.controlengineering.co.nz/Wikis/OPTI/pmwiki.php/Solvers/OOQP}{https://www.controlengineering.co.nz/Wikis/OPTI/pmwiki.php/Solvers/OOQP}}. 

(ii) The hybrid acceleration steps \eqref{xines} and \eqref{yines} unify two existing acceleration techniques. When $\alpha>0$, the hybrid acceleration reduces to an inertial extrapolation \cite{AA01} for the line search step. Especially for $\rho<0.5$, the restriction $(0, 1/\rho-1)$ is larger than the range $(0,1)$ in the literature \cite{DLL17,CCM15,PS17,GCH20,WSL23}. When $\alpha<0$, the hybrid acceleration reduces to a back substitution \cite{HTY12} for the subproblem step.

(iii) The inclusion of proximal terms in \eqref{xQP} and \eqref{yQP} ensures that the parameters $r$ and $s$ in the multiplier updates \eqref{lambda1} and \eqref{lambda2} can both be positive, enhancing the flexibility for ALDA.

(iv) The calculation of the matrices in Step 6 significantly influences both the theoretical properties and numerical performance of the HAP-PRS-SQP algorithm. For various approaches to computing these matrices, refer to \cite{JZY20,JLY21}. 
}\end{remark}

\section{Convergence analysis}\label{sec4}
This section analyzes the convergence of the HAP-PRS-SQP algorithm. Some notations and necessary assumptions are given below:
\[
\hat{w}_k=(w_k,d_{k-1}^y),~d_k=({d}_{k}^x,{d}_{k}^y),~t_k=(t_{k}^x,t_{k}^y),~y_{-1}=\bar{y}_0=\tilde{y}_0=y_0,~\mathcal{H}_{-1}^y=\mathcal{H}_0^y.
\]

\begin{assumption}\label{ass1}{\rm
		
		(i) The iteration sequence $\{w_k\}$ and matrix sequences $\{{H}_k^x\}$ and $\{{H}_k^y\}$ generated by the HAP-PRS-SQP algorithm are bounded, with $\|{H}_k^x\|\leq \eta^x$ and $\|{H}_k^y\|\leq \eta^y$ for all $k$.
		
		(ii) There exist lower bounds $\underline{\lambda}^x$ and $\underline{\lambda}^y$ of $\{\lambda_{\min}(H^x_k)\}$ and $\{\lambda_{\min}(H^y_k)\}$, respectively, such that
		\begin{equation}\label{NF01} \eta^x_1:=\underline{\lambda}^x+\beta\lambda_{\min}(A^TA)+\ell>0,\
			\eta^y_1:=\underline{\lambda}^y+\beta+\sigma>0.
		\end{equation}
		
		(iii) The gradients $\nabla f(x)$ and $\nabla g(y)$ are local Lipschitz continuous.
	}
\end{assumption}

\begin{remark}\label{remark3}{\rm (Comments on Assumption \ref{ass1})
		
		(i) The boundedness of $\{w_k\}$, together with additional suitable conditions, ensures the existence of a KKT point for \eqref{P} (see Theorem \ref{the1}). 
		Furthermore, the boundedness of $\{H_k^x\}$ and $\{H_k^y\}$ is achievable with practical choices, such as $H_k^x=\nabla^2f(x_k)$ and $H_k^y=\nabla^2g(y_k)$. 
		
		(ii) First, for any eigenvalue $\lambda^x_k$ of $H_{k}^x$ with eigenvector $\xi^x_k$, it follows that $\lambda^x_k=[(\xi^x_k)^TH^x_k\xi^x_k]/\|\xi^x_k\|^2\geq -\|H^x_k\|\geq -\eta^x$. So both $\{\lambda_{\min}(H^x_k)\}$ and $\{\lambda_{\min}(H^y_k)\}$ possess finite lower bounds. 
		Second, in case that the minimum eigenvalues of $H_{k}^x$ and/or $H_{k}^y$ can be obtained easily, we choose $\underline{\lambda}^x$ and/or $\underline{\lambda}^y$ as $\underline{\lambda}^x=\inf_{k}\lambda_{\min}(H^x_k)$ and/or $\underline{\lambda}^y=\inf_{k}\lambda_{\min}(H^y_k)$. Otherwise, the low-cost way to choose $\underline{\lambda}^y$ or/and $\underline{\lambda}^x$ is $\underline{\lambda}^y=-\eta^y$ or/and $\underline{\lambda}^x=-\eta^x$.
		Third, for known $\underline{\lambda}^y$ and $\underline{\lambda}^x$ as well as $\beta$, condition \eqref{NF01} can be satisfied easily by selecting appropriate proximal parameters $\ell$ and $\sigma$. 
		Fourth, from the definitions of $\mathcal{H}_k^x, \ \mathcal{H}_k^y,\ \eta^x_1$ and $\eta^y_1$, it is easy to know that
		\begin{subequations}\label{F20}
			\begin{numcases}{}
				\|\mathcal{H}_k^x\|\leq \eta_2^x:=\eta^x+\beta\|A^TA\|+\ell,\ \|\mathcal{H}_k^y\|\leq \eta_2^y:=\eta^y+\beta+\sigma,\label{F20a}\\
				x^T\mathcal{H}_k^xx\geq \eta^x_1\|x\|^2,\   y^T\mathcal{H}_k^y y\geq \eta^y_1\|y\|^2,\ \forall\ x\in \mathbb{R}^{n_1},\ \forall\ y\in \mathbb{R}^{n_2},\ \forall\ k.\label{F20b}
			\end{numcases}
		\end{subequations}
		
		(iii) Unlike some references such as \cite{JZY20,JLY21,JZY22,JCT23,JMX24}, Assumption \ref{ass1} makes no positive definiteness requirements for either $(H^x_k,H^y_k)$ or $(H^x_k+\beta A^TA,H^y_k+\beta I_{n_2})$. This allows broader matrix choices, such as exact Hessians $H_{k}^x=\nabla ^2 f (x_{k})$ and $H_{k}^y=\nabla ^2 g (y_{k})$.
		
		(iv) Assumption \ref{ass1} (iii) ensures a positive lower bound for stepsize sequence $\{t_k\}$ (Lemma \ref{lem2}). This condition holds automatically if $f$ and $g$ are twice continuously differentiable.
		
}\end{remark}

We first analyze several basic characters of the HAP-PRS-SQP algorithm.
\begin{lemma}\label{lem1}
	(i) Both line searches \eqref{xlines} and \eqref{ylines} in HAP-PRS-SQP are well-defined.
	(ii) If $w_{k+1}-w_k=0$, $(x_{k+1},y_{k+1})$ is a KKT point of \eqref{P} with the multiplier $\lambda_{k+1}$.
\end{lemma}
\begin{proof}
	(i) If $d_k^x=0$ and $d_k^y=0$, then stepsizes $t_k^x=1$ and $t_k^y=1$ trivially satisfy the line search conditions \eqref{xlines} and \eqref{ylines}. Otherwise, suppose there exists a nonnegative integer $k_0$ and $d_{k_0}^x\ne 0$ such that
	\[
	L_{\beta}(x_{k_0}+\nu^{i_{k_0}}d_{k_0}^{x}, y_{k_0}, \lambda_{k_0}) > L_{\beta}(x_{k_0}, y_{k_0}, \lambda_{k_0})-\rho \nu^{i_{k_0}}\|d_{k_0}^{x}\|_{\mathcal{H}_{k_0}^x}^{2}.
	\]
	This implies
	\[
	\frac{L_{\beta}(x_{k_0}+\nu^{i_{k_0}}d_{k_0}^{x}, y_{k_0}, \lambda_{k_0})-L_{\beta}(x_{k_0}, y_{k_0}, \lambda_{k_0})}{\nu^{i_{k_0}}} >-\rho \|d_{k_0}^{x}\|_{\mathcal{H}_{k_0}^x}^{2}.
	\]
	Taking $i_{k_0}\to \infty$, then
	\[
	\nabla _x L_{\beta}(w_{k_0})^{\top}d_{k_0}^x \geq -\rho \|d_{k_0}^{x}\|_{\mathcal{H}_{k_0}^x}^{2}.
	\]
	From \eqref{xd}, $d_{k_0}^x\ne 0$, $\mathcal{H}_{k_0}^x\succ0$ and $\ell>0$, we have
	\[
	- \frac{1}{1+\alpha}\|d_{k_0}^{x}\|_{\mathcal{H}_{k_0}^x}^{2}
	=\nabla _x L_{\beta}(w_{k_0})^{\top}d_{k_0}^x \geq -\rho \|d_{k_0}^{x}\|_{\mathcal{H}_{k_0}^x}^{2}.
	\]
	This yields $\rho\geq \frac{1}{1+\alpha}$, which contradicts $\alpha\in [0, \frac{1}{\rho}-1)$. Thus, the line search in Step 1 (iii) is well-defined. A similar argument applies to the line search in Step 3 (iii).
	
	(ii) If $w_{k+1}-w_k=0$, then from \eqref{newx}, \eqref{newy} and $t_k>0$, it follows that $d_k=0$. By the definitions of $d_k^x$ and $d_k^y$, we have $x_{k+1}=x_{k}=\bar{x}_{k+1}$ and $y_{k+1}=y_{k}=\bar{y}_{k+1}$.
	From \eqref{sd-x} and \eqref{sd-y}, we further deduce $\tilde{x}_{k+1}=x_k$ and $\tilde{y}_{k+1}=y_k$.
	Using the multiplier updates \eqref{lambda1} and \eqref{lambda2}, we obtain
	\[
	\lambda_{k+1} = \lambda_{k+\frac{1}{2}}-s \beta(A x_{k+1}-y_{k+1})=\lambda_{k}-(r+s) \beta(A x_{k+1}-y_{k+1})+r\beta(y_{k}-y_{k+1}).
	\]
	Since $r+s\not=0$ and $\beta>0$, t follows that $Ax_k=y_k$ and $\lambda_{k+\frac{1}{2}}=\lambda_{k}$. Combining these with \eqref{xoc} and \eqref{yoc}, we conclude
	\[
	\nabla f(x_k)=A^{\top}\lambda_k,~\nabla g(y_k)=-\lambda_k,~Ax_k=y_k.
	\]
	Thus $w_k$ is a stationary point of \eqref{P}.
\end{proof}

\begin{lemma}\label{lemNA}
	If Assumption \ref{ass1} (i,ii) holds, then the sequences $\{d_k\}$, $\{(\tilde{x}_k,\tilde{y}_k)\}$ and $\{(\bar{x}_k,\bar{y}_k)\}$ generated by HAP-PRS-SQP are all bounded, and so is the sequence $\{\hat{w}_k\}$.
\end{lemma}
\begin{proof}
	If $x_k$ is feasible and $\tilde{x}_{k+1}$ is optimal for \eqref{xQP}, then the optimality  implies
	\[
	0\geq[\nabla f(x_{k})-A^{\top}(\lambda_{k}-\beta(A x_{k}-y_{k}))]^{\top}(\tilde{x}_{k+1}-x_{k})+\frac{1}{2}\|\tilde{x}_{k+1}-x_{k}\|_{\mathcal{H}_k^x}^{2}.
	\]
	Since $\{w_k\}$ is bounded and $f(\cdot)$ is smooth, there exists a constant $G > 0$ such that
	\[\|\nabla f(x_{k})-A^{\top}(\lambda_{k}-\beta(A x_{k}-y_{k}))\|\leq G.\]
	Combining this with \eqref{sd-x} and \eqref{F20b}, we obtain
	\[\|d_{k}^{x}\| = \|\bar{x}_{k+1}-x_{k}\|=(1+\alpha)\|\tilde{x}_{k+1}-x_{k}\|\leq (1+\alpha)2G/\eta^x_1.\]
	This shows that $\{(d^x_k,\bar{x}_{k+1},\tilde{x}_{k+1})\}$ is bounded. A similar argument establishes the boundedness of $\{(d^y_k,\bar{y}_{k+1},\tilde{y}_{k+1})\}$. Consequently, the sequence $\{\hat{w}_k\}$ is bounded.
\end{proof}

From Lemma \ref{lemNA}, one knows that there exist two compact sets $\mathcal{S}^x$ and $\mathcal{S}^y$ in $ \mathbb{R}^{n_1}$ and $ \mathbb{R}^{n_2}$, respectively, such that
\begin{equation}\label{NF1}
	x_k+td^x_k\in \mathcal{S}^x, \   y_k+td^y_k\in \mathcal{S}^y,\ \forall\ t\in [0,1],\ \forall\ k.
\end{equation}
On the other hand, in view of the local Lipschitz continuity of $\nabla f(x)$ and $\nabla g(y)$, by Theorem 2.1.6 in \cite{CMN19}, we can conclude that $\nabla f(x)$ and $\nabla g(y)$ are Lipschitz continuous on $\mathcal{S}^x$ and $\mathcal{S}^y$, respectively. Namely, there exist two constants $L_f,L_g>0$ such that
\begin{equation}\label{NF5}
	\|\nabla f(x)-\nabla f(\tilde{x})\|\leq L_f\|x-\tilde{x}\|,\forall\ x,\tilde{x}\in \mathcal{S}^x;
	\|\nabla g(y)-\nabla g(\tilde{y})\|\leq L_g\|y-\tilde{y}\|,\forall\ y,\tilde{y}\in \mathcal{S}^y.
\end{equation}

The following lemma shows the infimum of the stepsize sequence $\{t_k\}$.
\begin{lemma}\label{lem2}
	If Assumption \ref{ass1}  holds, then the sequence $\{t_k\}$ yielded by the HAP-PRS-SQP algorithm possess a positive infimum independent of $k$ as follows:
	\[\label{stepL}
	\min \{t_k^x,t_k^y\} \geq \nu\min \left\{1,\frac{(\frac{1}{1+\alpha}-\rho)\eta_1^{x}}
	{L_f+\beta\|A^{\top}A\|},\frac{(\frac{1}{1+\alpha}-\rho)\eta_1 ^y}{L_g+\beta}\right\}:=\gamma,
	\]
	where $\rho,\nu\in (0,1)$ and $\alpha \in (-1, \frac{1}{\rho}-1)$.
\end{lemma}
\begin{proof} From \eqref{F2}, \eqref{NF1} and \eqref{NF5}, it follows that for all $t$, $\theta \in[0,1]$ and for all $k$,
	\[\left\{
	\begin{array}{l}
		\|\nabla_{x} L_{\beta}(x_{k}+\theta t d_{k}^{x}, y_{k}, \lambda_{k})-\nabla_{x } L_{\beta}(x_{k}, y_{k}, \lambda_{k})\| \leq(L_{f}+\beta\|A^{\top} A\|) t\|d_{k}^{x}\|,\\
		\|\nabla_{y} L_{\beta}(x_{k+1}, y_{k}+\theta t d_{k}^{y}, \lambda_{k+\frac{1}{2}})-\nabla_{y} L_{\beta}(x_{k+1}, y_{k}, \lambda_{k+\frac{1}{2}})\| \leq(L_{g}+\beta)t\|d_{k}^{y}\|.
	\end{array}
	\right.\]
	By the mean value theorem, the above inequalities, \eqref{xd} and \eqref{F20b}, for all $k \geq 0$ and any $t \in\left[0,\min\left\{1,\frac{(\frac{1}{1+\alpha}-\rho)\eta_1^{x}}{L_f+\beta\|A^{\top}A\|}\right\}\right]$, there exists $\theta_k^t \in (0,1)$ such that
	\[\begin{aligned}
		&L_{\beta}(x_k+td_k^x,y_k,\lambda_k)-L_{\beta}(x_k,y_k,\lambda_k)+\rho t\|d_k^x\|_{\mathcal{H}_k^x}^2\\
		=&t \nabla_{x} L_{\beta}(x_{k}+\theta_{k}^{t} t d_{k}^{x}, y_{k}, \lambda_{k})^{\top} d_{k}^{x}+\rho t\|d_{k}^{x}\|_{\mathcal{H}_{k}^{x}}^{2}\\
		=&t[\nabla_{x} L_{\beta}(x_{k}+\theta_{k}^{t} t d_{k}^{x}, y_{k}, \lambda_{k})-\nabla_{x} L_{\beta}(x_{k}, y_{k}, \lambda_{k})]^{\top} d_{k}^{x}+t \nabla_{x} L_{\beta}(x_{k}, y_{k}, \lambda_{k})^{\top} d_{k}^{x}
		+\rho t\|d_{k}^{x}\|_{\mathcal{H}_{k}^{x}}^{2}\\
		\leq&(L_{f}+\beta\|A^{\top} A\|) t^{2}\|d_{k}^{x}\|^{2}-\frac{t}{1+\alpha}\|d_{k}^{x}\|_{\mathcal{H}_{k}^{x}}^{2}
		+\rho t\|d_{k}^{x}\|_{\mathcal{H}_{k}^{x}}^{2} \\
		\leq &t[(L_{f}+\beta\|A^{\top} A\|) t-(\frac{1}{1+\alpha}-\rho)\eta_1^{x}]\|d_{k}^{x}\|^{2} \\
		\leq & 0.
	\end{aligned}\]
	Thus, it follows from the line search \eqref{xlines} that $t_k^x \geq \nu\min\left\{1,\frac{(\frac{1}{1+\alpha}-\rho)\eta_1^{x}}{L_f+\beta\|A^{\top}A\|}\right\}$. Similarly, we obtain $t_k^y \geq \nu\min \left\{1,\frac{(\frac{1}{1+\alpha}-\rho)\eta_1 ^y}{L_g+\beta}\right\}$. This completes the proof.
\end{proof}

Next, we present a lemma that plays a key role in the theoretical analysis of HAP-PRS-SQP. The lemma also elucidates why a break in the direction of the dual update does not fundamentally affect the theoretical properties of the unified framework. A new merit function $\hat{L}_{\beta}$ is considered as follows:
\[\hat{L}_{\beta}(\hat{w}_k):={L}_{\beta}(w_k)+\frac{6}{|r+s|\beta}\left(L_g^2+\beta ^2+(\frac{\eta_2^y}{1+\alpha})^2\right)\|d^y_{k-1}\|^2.\]

\begin{lemma}\label{lem3}
	Suppose that Assumption \ref{ass1} holds. Then
	\[\hat{L}_{\beta}(\hat{w}_{k+1})-\hat{L}_{\beta}(\hat{w}_{k}) \leq -\delta^x\|d_{k}^x\|^2-\delta^y\|d_{k}^y\|^2,\]
	where $\delta^x=\rho\gamma\eta^x_1-\frac{6(1-s)^2\beta\lambda_{\max}(A^{\top}A)}{|r+s|}$, $\delta^y =\rho\gamma\eta^y_1 -\frac{6}{|r+s|\beta}\left(L_{g}^{2}+(1+s^{2})\beta^{2}+2\left(\frac{\eta_2^y}{1+\alpha}\right)^2\right)-\frac{|rs|\beta}{|r+s|}$.
\end{lemma}
\begin{proof} From the definition of the ALF \eqref{alf}, and the multiplier updates \eqref{lambda1} and \eqref{lambda2}, we have
	\[
	L_{\beta}(w_{k+1})-L_{\beta}(x_{k+1}, y_{k+1}, \lambda_{k+\frac{1}{2}})=s \beta\|A x_{k+1}-y_{k+1}\|^{2},\]
	\[L_{\beta}(x_{k+1}, y_{k}, \lambda_{k+\frac{1}{2}})-L_{\beta}(x_{k+1}, y_{k}, \lambda_{k})=r \beta\|A x_{k+1}-y_{k}\|^{2}.\]
	From \eqref{F20b}, and the line search conditions \eqref{xlines} and \eqref{ylines}, we obtain
	\[
	L_{\beta}(x_{k+1}, y_{k}, \lambda_{k})-L_{\beta}(x_{k}, y_{k}, \lambda_{k})\leq-\rho t_k^x\eta^x_1 \|d_{k}^{x}\|^{2}, 
	\]
	\[
	L_{\beta}(x_{k+1}, y_{k+1}, \lambda_{k+\frac{1}{2}})-L_{\beta}(x_{k+1}, y_{k}, \lambda_{k+\frac{1}{2}}) \leq- \rho t_k^y\eta^y_1 \|d_{k}^{y}\|^{2}.
	\]
	Combining these four inequalities, we derive:
	\begin{align}
		{L}_{\beta}({w}_{k+1})-{L}_{\beta}({w}_{k}) &\leq -\rho t_k^x\eta^x_1 \|d_{k}^{x}\|^{2}-\rho t_k^y\eta^y_1 \|d_{k}^{y}\|^{2}\nonumber\\
		&\ \ \ \ +s \beta\|A x_{k+1}-y_{k+1}\|^{2}+r \beta\|A x_{k+1}-y_{k}\|^{2}.\label{meritL}
	\end{align}
	From \eqref{lambda1} and \eqref{lambda2}, the multiplier update can be expressed as
	\[\lambda_{k+1}=\lambda_k-(r+s)\beta(Ax_{k+1}-y_k)-s\beta(y_k-y_{k+1}).\]
	This implies
	\begin{align}
		A x_{k+1}-y_{k}&=-\frac{1}{(r+s) \beta}(\lambda_{k+1}-\lambda_{k})-\frac{s}{r+s}(y_{k}-y_{k+1}), \nonumber\\
		A x_{k+1}-y_{k+1}&=-\frac{1}{(r+s) \beta}(\lambda_{k+1}-\lambda_{k})+\frac{r}{r+s}(y_{k}-y_{k+1}). \label{lam2}
	\end{align}
	Using these relations, we bound the terms involving $A x_{k+1}-y_{k}$ and $A x_{k+1}-y_{k+1}$,
	\begin{align}
		&s \beta\|A x_{k+1}-y_{k+1}\|^{2}+r \beta\|A x_{k+1}-y_{k}\|^{2}\nonumber\\
		\leq&\frac{1}{|r+s|\beta}\|\lambda_{k+1}-\lambda_{k}\|^2+\frac{|rs|\beta}{|r+s|}\|y_{k+1}-y_{k}\|^2.\label{sr}
	\end{align}
	From \eqref{yoc} and \eqref{lambda2}, we have
	\begin{equation}\label{lambdak1}
		\lambda_{k+1}=-\nabla g (y_{k})+(1-s)\beta(A x_{k+1}-y_{k+1})+\beta(y_{k+1}-y_k)-\mathcal{H}_k^y (\tilde{y}_{k+1}-y_{k}).
	\end{equation}
	Using \eqref{sd-y}, \eqref{F20a}, \eqref{NF1} and \eqref{NF5}, we bounds $\|\lambda_{k+1}-\lambda_{k}\|$ as
	\begin{align}
		\|\lambda_{k+1}-\lambda_{k}\|
		&= \|-\nabla g(y_{k})+\nabla g(y_{k-1})+(1-s)\beta A(x_{k+1}-x_{k})-s\beta(y_{k}-y_{k+1})\nonumber\\
		&\ \ \ \ -\beta(y_k-y_{k-1})-\mathcal{H}_k^y (\tilde{y}_{k+1}-y_{k})+\mathcal{H}_{k-1}^y (\tilde{y}_{k}-y_{k-1})
		\| \nonumber\\
		&\leq L_g\|y_{k}-y_{k-1}\|+|1-s|\beta\|A(x_{k+1}-x_{k})\|+|s|\beta\|y_{k+1}-y_{k}\|\nonumber\\
		&\ \ \ \ +\beta\|y_k -y_{k-1}\| +\frac{\eta_2^y}{1+\alpha}\|d_k^y\|+\frac{\eta_2^y}{1+\alpha}\| d_{k-1}^y\|. \label{lk1lk}
	\end{align}
	Applying the Cauchy inequality and using \eqref{newx}, \eqref{newy} and $t_k\leq 1$, we further obtain
	\begin{align}
		\frac{1}{6}\|\lambda_{k+1}-\lambda_{k}\|^{2}
		&\leq(L_{g}^{2}+\beta^2)\|y_{k}-y_{k-1}\|^{2}+(1-s)^{2}\beta^{2}\|A(x_{k+1}-x_{k})\|^{2}\nonumber\\
		&\ \ \ \ +s^{2}\beta^{2}\|y_{k+1}-y_{k}\|^{2}  +\left(\frac{\eta_2^y}{1+\alpha}\right)^2\|d_k^y\|^2+\left(\frac{\eta_2^y}{1+\alpha}\right)^2\|d_{k-1}^y\|^2 \nonumber\\
		&\leq(L_{g}^{2}+\beta^2)\|d_{k-1}^y\|^{2}+(1-s)^{2}\beta^{2}\lambda_{\max}(A^{\top}A)\|d_{k}^x\|^{2}\nonumber\\
		&\ \ \ \ + s^{2}\beta^{2}\|d_{k}^y\|^{2}+\left(\frac{\eta_2^y}{1+\alpha}\right)^2\|d_k^y\|^2+\left(\frac{\eta_2^y}{1+\alpha}\right)^2\|d_{k-1}^y\|^2. \label{LK1LK}
	\end{align}
	Substituting \eqref{LK1LK} into \eqref{sr} and using $t_k^y\leq 1$, we derive
	\[\begin{aligned}
		&s \beta\|A x_{k+1}-y_{k+1}\|^{2}+r \beta\|A x_{k+1}-y_{k}\|^{2} \\
		\leq &\frac{6}{|r+s|\beta}\left[(L_{g}^{2}+\beta^2)\|d_{k-1}^y\|^{2}
		+(1-s)^{2}\beta^{2}\lambda_{\max}(A^{\top}A)\|d_{k}^x\|^{2}+s^{2}\beta^{2} \|d_{k}^y\|^{2}\right. \\
		&+\left.\left(\frac{\eta_2^y }{1+\alpha}\right)^2\|d_k^y\|^2
		+\left(\frac{\eta_2^y }{1+\alpha}\right)^2\|d_{k-1}^y\|^2\right]
		+\frac{|rs|\beta }{|r+s|}\|y_{k+1}-y_{k}\|^2 \\
		\leq& \frac{6}{|r+s|\beta}\left[\left(L_{g}^{2}+\beta^2+
		\left(\frac{\eta_2^y }{1+\alpha}\right)^2\right)\|d_{k-1}^y\|^{2}
		+(1-s)^{2}\beta^{2}\lambda_{\max}(A^{\top}A)\|d_{k}^x\|^{2}\right] \\
		&+\left[\frac{6}{|r+s|\beta}\left(s^{2}\beta^{2}
		+\left(\frac{\eta_2^y }{1+\alpha}\right)^2\right)
		+\frac{|rs|\beta}{|r+s|}\right]\|d_{k}^y\|^2. 
	\end{aligned}\]
	Substituting the above result into \eqref{meritL}, we obtain
	\[\begin{aligned}
		&{L}_{\beta}({w}_{k+1})-{L}_{\beta}({w}_{k}) \nonumber\\
		\leq & -\rho t_k^x\eta^x_1\|d_{k}^{x}\|^{2}-\rho t_k^y\eta^y_1\|d_{k}^{y}\|^{2}+\left[\frac{6}{|r+s|\beta }\left(s^{2}\beta^{2}
		+\left(\frac{\eta_2^y}{1+\alpha}\right)^2\right)+\frac{|rs|\beta}{|r+s|}\right]\|d_{k}^y\|^2\nonumber\\
		& +\frac{6}{|r+s|\beta}\left[\left(L_{g}^{2}+\beta^2
		+\left(\frac{\eta_2^y}{1+\alpha}\right)^2\right)\|d_{k-1}^y\|^{2}
		+(1-s)^{2}\beta^{2}\lambda_{\max}(A^{\top}A)\|d_{k}^x\|^{2}\right].\label{lw1lw}
	\end{aligned}\]
	Finally, using Lemma \ref{lem2} $(t_k\geq\gamma)$, this completes the proof.
\end{proof}

By lemma \ref{lem3}, the sequence ${\hat{L}_{\beta}(\hat{w}_{k})}$ has a sufficient descent property provided that $\delta^x > 0$ and $\delta^y > 0$. These conditions can be satisfied by appropriately choosing the parameters $(r, s, \beta, \ell, \sigma)$. {\it In this case, the dual update direction does not significantly change the monotonicity of the underlying function $\hat{L}_{\beta}$, thus ensuring the convergence of the sequence ${\hat{L}_{\beta}(\hat{w}_k)}$, as detailed in Theorem \ref{the1} (i).}

\begin{assumption}\label{ass3}{\rm Suppose that the parameters in the HAP-PRS-SQP algorithm are chosen such that
$\delta^x> 0$ and $\delta^y> 0$. 
}\end{assumption}
To ensure $\delta^x > 0$ and $\delta^y > 0$, it suffices to choose the specific parameter selections as follows: $0<\rho<1$, $-1<\alpha<\frac{1}{\rho}-1$, $r+s\ne0$, $\beta>0$,  $\ell>\frac{6(1-s)^2\beta\lambda_{\max}(A^{\top}A)}{|r+s|\rho\gamma}-\underline{\lambda}^x-\beta\lambda_{\min}(A^TA)$ and $\sigma >\frac{6}{|r+s|\beta\rho\gamma}\left(L_{g}^{2}+(1+s^{2}+\frac{|rs|}{6})\beta^{2}+2\left(\frac{\eta_2^y}{1+\alpha}\right)^2\right)-\underline{\lambda}^y-\beta$.
Specifically, we give two settings to verify that two dual update in HAP-PRS-SQP algorithm can both be ascent or descent. Here constants are not optimized to allow for flexibility in the precise definition of parameters. 
Let $\beta>0$ be given. Choose the proximal parameters $\ell$ and $\sigma$ such that $\ell>-\underline{\lambda}^x-\beta\lambda_{\min}(A^TA)$ and $\sigma>-\underline{\lambda}^y-\beta$, which guarantees $\eta^x_1=\underline{\lambda}^x+\beta\lambda_{\min}(A^TA)+\ell>0$ and $\eta^y_1=\underline{\lambda}^y+\beta+\sigma>0$. By Lemma \ref{lem2}, for these choices we have $\gamma>0$.

$\bullet$ Dual-ascent case (ALDA). Set $s=1$ (ascent in the second dual update). Then $\delta^x=\rho\gamma\eta^x_1>0$ holds automatically. To ensure that $\delta^y>0$, choose the first dual stepsize as
\[
r= \frac{6(L_{g}^{2}+2\beta^{2}+2(\frac{\eta_2^y}{1+\alpha})^2)}
{\beta\rho\gamma\eta^y_1}>0,
\]
Substituting this $r$ into the definition of $\delta^y$ with $s=1$ gives
\[\delta^y =\rho\gamma\eta^y_1 -\frac{6}{|r+1|\beta}\left(L_{g}^{2}+2\beta^{2}+2\left(\frac{\eta_2^y}{1+\alpha}\right)^2\right)-\frac{|r|\beta}{|r+1|}.\]
Using the chosen $r$, one verifies that 
\[
	\rho\gamma\eta^y_1 =\frac{6}{r\beta}(L_{g}^{2}+2\beta^{2}+2(\frac{\eta_2^y}{1+\alpha})^2)
	>\frac{6}{|r+1|\beta}(L_{g}^{2}+2\beta^{2}+2(\frac{\eta_2^y}{1+\alpha})^2)-\frac{|r|\beta}{|r+1|},
\]
and consequently $\delta^y>0$. Hence both $\delta^x$ and $\delta^y$ are positive in this ascent-ascent configuration.

$\bullet$ Dual-descent case (ALDD). Choosing $r<0$ and $s<0$ (descent in both dual updates). In this case, from \eqref{meritL} and Lemma \ref{lem2} we obtain
\begin{align}
	L_{\beta}({w}_{k+1})-L_{\beta}({w}_{k}) \overset{\eqref{meritL}}{\leq} &-\rho t_k^x\eta^x_1 \|d_{k}^{x}\|^{2}-\rho t_k^y\eta^y_1 \|d_{k}^{y}\|^{2}\nonumber\\
	&+s \beta\|A x_{k+1}-y_{k+1}\|^{2}+r \beta\|A x_{k+1}-y_{k}\|^{2}\nonumber\\
	\leq &-\rho \gamma\eta^x_1 \|d_{k}^{x}\|^{2}-\rho \gamma\eta^y_1 \|d_{k}^{y}\|^{2}.\label{ALDD-merit-dec}
\end{align}
Thus the augmented Lagrangian $L_{\beta}(w_{k})$ itself is monotonically non-increasing. Consequently, the convergence analysis can be carried out directly on $L_{\beta}(w_k)$ instead of the extended merit function $L_{\beta}(\hat{w}_k)$. A concise outline of this simplified proof is given in Remark \ref{remark4} at the end of this section.

We give some sufficient conditions to guarantee the sequence $\{w_k\}$ generated by HAP-PRS-SQP is bounded. It is important for the forthcoming convergence analysis.

\begin{lemma}
	Suppose Assumption \ref{ass3} holds and $\beta$ is further chosen so that there exists a constant $\frac{1}{3\beta}<\tau<\frac{2}{|r+s|\beta}\left(L_g^2+\beta ^2+(\frac{\eta_2^y}{1+\alpha})^2\right)/ (\beta+\frac{\eta_2^y}{1+\alpha})^2$ with
	\begin{equation}\label{lem:boundseq}
		\hat{g}:=\inf_y \left\{g(y)-3\tau \|\nabla g (y)+(1-s)\beta(A x-y)\|^2\right\}>-\infty.
	\end{equation}
		Moreover, assume that $\lim\inf_{\|x\|\to\infty}f(x)=\infty$ and that the gradients of $f(x)$ and $g(y)$ are Lipschitz continuous. Then the sequence $\{w_k\}$ generated by HAP-PRS-SQP is bounded.
\end{lemma}
\begin{proof}
	If the gradients of $f(x)$ and $g(y)$ are Lipschitz continuous, then for any $x_1,x_2\in \mathbb{R}^{n_1}$ and $y_1,y_2\in \mathbb{R}^{n_2}$, there exist two constants $L_f,L_g>0$ such that
	\[
		\|\nabla f(x_1)-\nabla f(x_2)\|\leq L_f\|x_1-x_2\|,
		\|\nabla g(y_1)-\nabla g(y_2)\|\leq L_g\|y_1-y_2\|.
	\]
	Thus, the same lower bound for the linesearch stepsizes $\{t_k\}$ as in Lemma \ref{lem2} holds, and Lemma \ref{lem3} remains valid. Together with Assumption \ref{ass3}, this implies that $\{\hat{L}_{\beta}(\hat{w}_k)\}$ is decreasing. From \eqref{lambdak1} and \eqref{sd-y}, we have
	\[\begin{aligned}
		\|\lambda_k\|^2&=\|-\nabla g (y_{k-1})+(1-s)\beta(A x_{k}-y_{k})+\beta(y_{k}-y_{k-1})-\mathcal{H}_{k-1}^y (\tilde{y}_{k}-y_{k-1})\|^2\\
		&\leq 3\|\nabla g (y_{k-1})+(1-s)\beta(A x_{k}-y_{k})\|^2+3(\beta+\frac{\eta_2^y}{1+\alpha})^2\|d_{k-1}^{y}\|^2.
	\end{aligned}\]
	Therefore,
	\[\begin{aligned}
		&\hat{L}_{\beta}(\hat{w}_0)\geq \hat{L}_{\beta}(\hat{w}_k)\\
		=&f(x_k)+g(y_k)-\lambda_k^{\top}(Ax_k-y_k)+\frac{\beta}{2}\|Ax_k-y_k\|^2+\frac{6}{|r+s|\beta}\left(L_g^2+\beta ^2+(\frac{\eta_2^y}{1+\alpha})^2\right)\|d^y_{k-1}\|^2\\
		=&f(x_k)+g(y_k)-\frac{1}{2\beta}\|\lambda_k\|^2+\frac{\beta}{2}\|Ax_k-y_k-\frac{\lambda_k}{\beta}\|^2+\frac{6}{|r+s|\beta}\left(L_g^2+\beta ^2+(\frac{\eta_2^y}{1+\alpha})^2\right)\|d^y_{k-1}\|^2\\
		=&f(x_k)+g(y_k)-\tau\|\lambda_k\|^2+(\tau-\frac{1}{2\beta})\|\lambda_k\|^2+\frac{\beta}{2}\|Ax_k-y_k-\frac{\lambda_k}{\beta}\|^2\\
		&+\frac{6}{|r+s|\beta}\left(L_g^2+\beta ^2+(\frac{\eta_2^y}{1+\alpha})^2\right)\|d^y_{k-1}\|^2\\
		\geq&f(x_k)+g(y_k)-3\tau\|\nabla g (y_{k-1})+(1-s)\beta(A x_{k}-y_{k})\|^2+(\tau-\frac{1}{2\beta})\|\lambda_k\|^2+\frac{\beta}{2}\|Ax_k-y_k-\frac{\lambda_k}{\beta}\|^2\\
		&+\left(\frac{6}{|r+s|\beta}\left(L_g^2+\beta ^2+(\frac{\eta_2^y}{1+\alpha})^2\right)-3\tau (\beta+\frac{\eta_2^y}{1+\alpha})^2\right)\|d^y_{k-1}\|^2\\
		\geq&f(x_k)+\hat{g}+(\tau-\frac{1}{2\beta})\|\lambda_k\|^2+\frac{\beta}{2}\|Ax_k-y_k-\frac{\lambda_k}{\beta}\|^2\\
		&+\left(\frac{6}{|r+s|\beta}\left(L_g^2+\beta ^2+(\frac{\eta_2^y}{1+\alpha})^2\right)-3\tau (\beta+\frac{\eta_2^y}{1+\alpha})^2\right)\|d^y_{k-1}\|^2.
	\end{aligned}\]
The condition $\lim\inf_{\|x\|\to\infty}f(x)=\infty$ implies that $\inf_x f(x)>-\infty$. Hence the above inequality guarantees the boundedness of the sequences $\{x_k\}$, $\{\lambda_k\}$, $\{Ax_k-y_k-\frac{\lambda_k}{\beta}\}$ and $\{d^y_{k-1}\}$. Consequently, the sequence $\{y_k\}$ is also bounded. 
\end{proof}

\begin{remark}
	(i) We illustrate condition \eqref{lem:boundseq} on a common problem:
	\[\min_x \frac{1}{2}\|Ax-b\|^2+\mu \|x\|^2.\]
	Set $f(x)=\mu \|x\|^2$ and $g(Ax)=\frac{1}{2}\|Ax-b\|^2$, so that $L_g=\lambda_{\max}(A^\top A)$. Choose $\beta=3L_g$ and $|r+s|=2\left(10L_g^2+(\frac{\eta_2^y}{1+\alpha})^2\right)/ (3L_g+\frac{\eta_2^y}{1+\alpha})^2$. Then there exists $\tau = \frac{1}{6L_g}$ satisfying $\frac{1}{9L_g} < \tau < \frac{1}{3L_g}$, and one verifies
	\[g(y)-3\tau \|\nabla g (y)+(1-s)\beta(A x-y)\|^2=\frac{1}{2}\left(\|Ax-b\|^2-\frac{1}{L_g}\|A^\top(Ax-b)\|^2\right)\geq 0,\]
	so that \eqref{lem:boundseq} holds with $\hat{g} = 0$.
	
	(ii) As discussed in \cite{WYZ19}, a weaker sufficient condition is the coercivity of $f(x)+g(y)$ on the feasible set $\mathcal{F}:= \{(x, y) \in \mathbb{R}^{n_1+n_2}|Ax - y = 0\}$, i.e., $f (x) + g(y) \to \infty$ whenever $(x, y)\in \mathcal{F}$ and $(x, y)\to \infty$.  This directly implies the boundedness of $\{(x, y)\}$. Together with \eqref{lambdak1} and the Lipschitz continuity of $\nabla g$, this yields the boundedness of $\{\lambda_k\}$ and hence of $\{w_k\}$.
\end{remark}

We now establish the subsequential convergence in Theorem \ref{the1} and the whole sequential convergence in Theorem \ref{the2} for HAP-PRS-SQP. Here, we focus on the $\hat{L}_{\beta}(\hat{w}_k)$, providing a general and unified analysis. 
\begin{theorem}\label{the1}(Global convergence) Let $\Omega$ and $\hat{\Omega}$ denote the cluster point set of sequences $\{w_{k}\}$ and $\{\hat{w}_{k}\}$, respectively. Suppose that Assumptions \ref{ass1} and \ref{ass3} hold. Then the five claims for the HAP-PRS-SQP algorithm as follows hold true.

(i) The whole sequence $\{\hat{L}_{\beta}(\hat{w}_k)\}$ is convergent, and $\hat{L}_{\beta}(\hat{w}_*)=\lim_{k \to \infty }\hat{L}_{\beta}(\hat{w}_k)=\inf_{k}\hat{L}_{\beta}(\hat{w}_k)$, $\forall~\hat{w}_* \in \hat{\Omega}$. Therefore, $\hat{L}_{\beta}(\cdot)$ is finite and constant on $\hat{\Omega}$.

(ii) $ \lim_{k\to\infty}\|d_k^x\|=\lim_{k\to\infty}\|d_k^y\|=0,~\sum_{k=0}^{+\infty}\| w_{k+1}-w_k \|^2<+\infty$.

(iii) $\Omega$ and $\hat{\Omega}$ are nonempty compact sets, and  $d(w_k, \Omega) \to 0$ and $d(\hat{w}_k, \hat{\Omega}) \to 0$ as $k \to \infty$.

(iv) $\Omega \subseteq {\rm crit}~L_{\beta}$, so the HAP-PRS-SQP algorithm is global convergent.

(v) $\hat{\Omega} = \{(x_*,y_*,\lambda_*,0):~(x_*,y_*,\lambda_*) \in \Omega\}=\Omega\times {0}$.
\end{theorem}
\begin{proof}
(i) By Lemma \ref{lemNA}, the sequence $\{\hat{w}_{k}\}$ is bounded, and so is $\{\hat{L}_{\beta}(\hat{w}_{k})\}$. Combined with the monotonic descent of $\{\hat{L}_{\beta}(\hat{w}_{k})\}$ (Lemma \ref{lem3}), it follows that $\{\hat{L}_{\beta}(\hat{w}_{k})\}$ is convergent. For any $\hat{w}^*\in \hat{\Omega}$, there exists a subsequence $\{\hat{w}_{k_j}\}$ such that $\hat{w}_{k_j} \to \hat{w}_*$, $k_j\to\infty$. Thus,
$\hat{L}_{\beta}(\hat{w}_*)=\lim_{k_j \to \infty }\hat{L}_{\beta}(\hat{w}_{k_j})=\lim_{k \to \infty }\hat{L}_{\beta}(\hat{w}_{k})=\inf_{k}\hat{L}_{\beta}(\hat{w}_k)$.

(ii) From Lemma \ref{lem3}, we have
\[	\delta^x\|d_{k}^x\|^2+\delta^y\|d_{k}^y\|^2\leq\hat{L}_{\beta}(\hat{w}_{k}) -\hat{L}_{\beta}(\hat{w}_{k+1}).\]
Summing this inequality for $k=0,1, \dots, p$ and noticing that $\hat{L}_{\beta}(\hat {w}_{0})<+\infty$, we obtain
\[	
\sum_{k=0}^p(\delta^x\|d_{k}^x\|^2+\delta^y\|d_{k}^y\|^2)\leq\hat{L}_{\beta}(\hat{w}_{0}) -\hat{L}_{\beta}(\hat{w}_{p+1})
\leq\hat{L}_{\beta}(\hat{w}_{0}) -\hat{L}_{\beta}(\hat{w}_{*})<+\infty.
\]
Since $\delta^x>0$ and $\delta^y>0$, it follows that
\[\sum_{k=0}^{+\infty} \|d_k^x\|^2<+\infty,~\sum_{k=0}^{+\infty} \|d_k^y\|^2<+\infty.\]
Using \eqref{LK1LK}, we further deduce
\[
\sum_{k=0}^{+\infty} \|{\lambda}_{k+1}-\lambda_{k}\|^2<+\infty;\ \lim_{k\to\infty}\|d_k^x\|=\lim_{k\to\infty}\|d_k^y\|=0.
\]
Considering that
\begin{equation}\label{xleqd}
\|x_{k+1}-x_k\|\leq \frac{1}{t_k^x}\|{x}_{k+1}-x_{k}\|=\|d_k^x\|,\ \|y_{k+1}-y_k\|\leq \frac{1}{t_k^y}\|{y}_{k+1}-y_{k}\|=\|d_k^y\|.
\end{equation}
Thus,
\[\sum_{k=0}^{+\infty}\|x_{k+1}-x_k\|^2<+\infty,~\sum_{k=0}^{+\infty}\|y_{k+1}-y_k\|^2<+\infty.\]
Consequently, $\sum_{k=0}^{+\infty}\|w_{k+1}-w_k\|^2<+\infty$.

(iii) The boundedness of $\{w_{k}\}$ and $\{\hat{w}_{k}\}$, along with the definitions of $\Omega$ and $\hat{\Omega}$, implies the claim directly.

(iv) For any given $w_*=(x_*,y_*,\lambda_*)\in\Omega$, there exists a subsequence $\{w_{k_j}\}$ such that $w_{k_j} \to w_*$, $k_j \to \infty$. Since $\|w_{k+1}-w_k\|\to 0$ as $k\to\infty$, it follows that $w_{k_j+1} \to w_*$, $k_j \to \infty$. From \eqref{lambda1}, $\{\lambda_{{k_j}+\frac{1}{2}}\}$ is convergent.
Let $\lambda_{{k_j}+\frac{1}{2}} \to \lambda_{**}$ as $k_j\to \infty$. Taking the limit $k={k_j}\to \infty$ in \eqref{lambda1} and \eqref{lambda2}, we obtain
\[\lambda_{**}=\lambda_{*}-r\beta(Ax_*-y_*),~\lambda_{*}=\lambda_{**}-s\beta(Ax_*-y_*).\]
Since $r+s\not=0$ and $\beta>0$, it follows that $Ax_*=y_*$ and $\lambda_{**}=\lambda_{*}$.
From \eqref{sd-x}, \eqref{sd-y} and conclusion (ii), we have
\[\lim\limits_{k_j\to \infty}\tilde{x}_{{k_j}+1}=\lim\limits_{k_j\to \infty}(x_{k_j}+\frac{1}{1+\alpha}d_{k_j}^x)=x_*,
\lim\limits_{k_j\to \infty}\tilde{y}_{{k_j}+1}=\lim\limits_{k_j\to \infty}(y_{k_j}+\frac{1}{1+\alpha}d_{k_j}^y)=y_*.\]
By the continuity of $\nabla f$ and $\nabla g$, and using $Ax_*=y_*$, taking the limit in \eqref{xoc} and \eqref{yoc} as $k=k_j\to \infty$, we obtain
\[
\nabla f(x_*)-A^{\top}\lambda_*=0,~\nabla g(y_*)+\lambda_*=0,~Ax_*-y_*=0.
\]
Thus, $w_*=(x_*,y_*,\lambda_*)$ satisfies the KKT conditions and belongs to crit $L_\beta$.

(v) By the definition of $\hat{w}_k$ and $d_k^y \to 0$ as $k \to \infty$, the claim follows directly.
\end{proof}

The following result  is a premise to prove the strong convergence of HAP-PRS-SQP.
\begin{lemma}\label{lem5}
If Assumptions \ref{ass1} and \ref{ass3} hold, then there exists a constant $c>0$ such that
\[\|\nabla_{\hat{w}} \hat{L}_{\beta}(\hat{w}_{k+1})\|\leq c(\|{x}_{k+1}-x_{k}\|+\|{y}_{k+1}-y_{k}\|+\|{y}_{k}-y_{k-1}\|).\]
\end{lemma}
\begin{proof}
From the definition of $\hat{L}_{\beta}$ and \eqref{F2}, we have
\begin{subequations}\label{NF41}
\begin{numcases}{}	\nabla_x \hat{L}_{\beta}(\hat{w}_{k+1}) = \nabla f(x_{k+1})-A^{\top}\lambda_{k+1}+\beta A^{\top}(A x_{k+1}-y_{k+1}),\label{NF41a}\\
	\nabla_y \hat{L}_{\beta}(\hat{w}_{k+1}) = \nabla g (y_{k+1})+\lambda_{k+1}-\beta(A x_{k+1}-y_{k+1}),\label{NF41b}\\
	\nabla_\lambda \hat{L}_{\beta}(\hat{w}_{k+1})= -(A x_{k+1}-y_{k+1}),\label{NF41c}\\
	\nabla_{d^y} \hat{L}_{\beta}(\hat{w}_{k+1}) = \frac{12}{|r+s|\beta}\left(L_g^2+\beta ^2+\left(\frac{\eta_2^y}{1+\theta}\right)^2\right)d^y_k.\label{NF41d}
\end{numcases}
\end{subequations}
Combining \eqref{NF41a} with \eqref{xoc}, we obtain
\[\begin{aligned}
\nabla_x \hat{L}_{\beta}(\hat{w}_{k+1})	=&\nabla f(x_{k+1})-A^{\top}\lambda_{k+1}+\beta A^{\top}(A x_{k+1}-y_{k+1})\\
&-\left[ \nabla f(x_{k})-A^{\top}\lambda_{k}+\beta A^{\top}(A x_{k}-y_{k})+\mathcal{H}_k^x(\tilde{x}_{k+1}-x_{k}) \right]\\
=&\nabla f(x_{k+1})-\nabla f(x_{k})-A^{\top}(\lambda_{k+1}-\lambda_{k})\\
&+\beta A^{\top}[A (x_{k+1}-x_{k})+y_{k}-y_{k+1}]-\mathcal{H}_k^x(\tilde{x}_{k+1}-x_{k}).
\end{aligned}\]
Using \eqref{NF1}, \eqref{NF5}, \eqref{F20a} and the fact that
\[
\|\tilde{x}_{k+1}-x_{k}\| \overset{\eqref{sd-x}}{=}\frac{ \|d_k^x\|}{1+\alpha}=\frac{\|t_k^xd_k^x\|}{(1+\alpha)t_k^x} 
\leq \frac{ \|{x}_{k+1}-x_{k}\|}{(1+\alpha)\gamma}=\mathcal{O}(\|{x}_{k+1}-x_{k}\|),
\]
we derive
\begin{equation}\label{pxl}
\|\nabla_x \hat{L}_{\beta}(\hat{w}_{k+1})\|\leq \mathcal{O}(\|{x}_{k+1}-x_{k}\|)+\mathcal{O}(\|{y}_{k+1}-y_{k}\|)+\mathcal{O}(\|{\lambda}_{k+1}-\lambda_{k}\|).
\end{equation}
From \eqref{NF41b} and \eqref{yoc}, we have
\[\begin{aligned}
\nabla_y \hat{L}_{\beta}(\hat{w}_{k+1})&=\nabla g (y_{k+1})+\lambda_{k+1}-\beta(A x_{k+1}-y_{k+1})\\
& \ \ \ -[\nabla g (y_{k})+\lambda_{k+\frac{1}{2}}-\beta(A x_{k+1}-y_{k})+\mathcal{H}_k^y(\tilde{y}_{k+1}-y_{k})]\\
& =\nabla g(y_{k+1})-\nabla g(y_{k})+\lambda_{k+1}-\lambda_{k+\frac{1}{2}}-\beta(y_{k}-y_{k+1})
-\mathcal{H}_k^y(\tilde{y}_{k+1}-y_{k}).
\end{aligned}\]
Using \eqref{sd-y}, we also have
\begin{equation}\label{tyk1yk}
\|\tilde{y}_{k+1}-y_{k}\| =\frac{ \|d_k^y\|}{1+\alpha}=\frac{\|t_k^y d_k^y\|}{(1+\alpha)t_k^y} 
\leq \frac{\|{y}_{k+1}-y_{k}\|}{(1+\alpha)\gamma} =\mathcal{O}(\|{y}_{k+1}-y_{k}\|).
\end{equation}
Combining these with \eqref{NF1}, \eqref{NF5} and \eqref{F20a}, we obtain  	
\[
\|\nabla_y \hat{L}_{\beta}(\hat{w}_{k+1})\|\leq\mathcal{O}(\|{y}_{k+1}-y_{k}\|)
+\mathcal{O}(\|{\lambda}_{k+1}-\lambda_{k+\frac{1}{2}}\|).
\]
From \eqref{lambda2} and \eqref{lam2}, we further bound $\|\lambda_{k+1}-\lambda_{k+\frac{1}{2}}\|$ as
\[
\|\lambda_{k+1}-\lambda_{k+\frac{1}{2}}\|=\|s\beta(Ax_{k+1}-y_{k+1})\|=\mathcal{O}(\|{y}_{k+1}-y_{k}\|)+\mathcal{O}(\|{\lambda}_{k+1}-\lambda_{k}\|).
\]
Thus, we conclude
\begin{equation}\label{NF45}
\|\nabla_y \hat{L}_{\beta}(\hat{w}_{k+1})\|\leq\mathcal{O}(\|{y}_{k+1}-y_{k}\|)
+\mathcal{O}(\|{\lambda}_{k+1}-\lambda_{k}\|).
\end{equation}
From \eqref{NF41c} and \eqref{lam2}, we have
\begin{equation}\label{NF47}
\|\nabla_\lambda \hat{L}_{\beta}(\hat{w}_{k+1})\|=\|Ax_{k+1}-y_{k+1}\|=\mathcal{O}(\|{y}_{k+1}-y_{k}\|)+\mathcal{O}(\|{\lambda}_{k+1}-\lambda_{k}\|).
\end{equation}
From \eqref{NF41d} and \eqref{tyk1yk}, we obtain
\begin{equation}\label{NF50}
\|\nabla_{d^y}\hat{L}_{\beta}(\hat{w}_{k+1})\|=\mathcal{O}(\|d^y_k\|)=\mathcal{O}(\|{y}_{k+1}-y_{k}\|).
\end{equation}
By \eqref{lk1lk} and \eqref{tyk1yk}, we derive
\begin{equation}\label{NF55}
\|{\lambda}_{k+1}-\lambda_{k}\|\leq\mathcal{O}(\|{x}_{k+1}-x_{k}\|)
+\mathcal{O}(\|{y}_{k+1}-y_{k}\|)+\mathcal{O}(\|{y}_{k}-y_{k-1}\|).
\end{equation}
Combining \eqref{pxl} and \eqref{NF45}-\eqref{NF55}, we conclude that the result holds.
\end{proof}

Next, we prove the whole sequential convergence of Algorithm \ref{algo1} using the KL property. For a function to be a KL function, it is sufficient that its components belong to a known KL function class. In problem \eqref{P}, the augmented Lagrangian $L_\beta(w)$ and the derived merit function $\hat{L}_\beta(\hat{w})$ are composed of $f$, $g$, and polynomial terms. Consequently, if $f$ and $g$ are KL functions themselves (e.g., semi-algebraic or subanalytic), then $L_\beta$ and $\hat{L}_\beta$ are also KL functions.

\begin{theorem}\label{the2} (Strong convergence)
Suppose that Assumptions \ref{ass1} and \ref{ass3} hold. If $f$ and $g$ are KL functions, then $\sum_{k=0}^{+\infty} \|{w}_{k+1}-w_{k}\|<+\infty$. Furthermore, the sequence $\{w_k\}$ of iterates converges to a stationary point $w_*$ of $L_{\beta}(\cdot)$, and $(x^*,y^*)$ is a KKT point of problem \eqref{P}.
\end{theorem}
\begin{proof} From Theorem \ref{the1}, we have $\lim\limits_{k \to \infty} \hat{L}_{\beta}(\hat{w}_{k})=\hat{L}_{\beta}(\hat{w}_{*})=\inf_{k} \hat{L}_{\beta}(\hat{w}_{k}),~\forall~\hat{w}_{*}\in \hat{\Omega}$.
We now consider two cases:

Case (i) (Finite termination): If there exists an integer $k_0$ such that $\hat{L}_{\beta}(\hat{w}_{k_0})=\hat{L}_{\beta}(\hat{w}_{*})$, then by Lemma \ref{lem3}, for all $k\geq k_0$,
\[
\delta^x\|d_{k}^x\|^2+\delta^y\|d_{k}^y\|^2
\leq \hat{L}_{\beta}(\hat{w}_{k+1})-\hat{L}_{\beta}(\hat{w}_{k}) \leq \hat{L}_{\beta}(\hat{w}_{k_0})-\hat{L}_{\beta}(\hat{w}_{*})=0.
\]
Since $\delta^x>0$ and $\delta^y>0$, it follows that $\|d_{k}^x\|=\|d_{k}^y\|=0$, for all $k\geq k_0$. From \eqref{LK1LK} and \eqref{xleqd}, we deduce ${\lambda}_{k+1}=\lambda_{k}$, ${x}_{k+1}=x_{k}$ and ${y}_{k+1}=y_{k}$, $\forall\ k\geq k_0$. Thus, ${w}_{k+1}=w_{k}$, for all $k\geq k_0$, and the claim holds.

Case (ii) (Infinite case): Assume $\hat{L}_{\beta}(\hat{w}_{k})>\hat{L}_{\beta}(\hat{w}_{*})$, for all $k$. By Theorem \ref{the1} (i) and Lemma \ref{KLP}, $\hat{L}_{\beta}(\cdot)$ satisfies the uniformized KL property with $(\epsilon,\eta,\varphi)$. Since $d(\hat{w}_k, \hat{\Omega}) \to 0$ and $\inf_k \hat{L}_{\beta}(\hat{w}_{k})=\hat{L}_{\beta}(\hat{w}_{*})$, for the given positive parameters $\epsilon$ and $\eta$ in Lemma \ref{KLP}, there exists an integer $\tilde{k}>0$ such that for all $k>\tilde{k}$,
\[
d(\hat{w}_k, \hat{\Omega}) <\epsilon,~\hat{L}_{\beta}(\hat{w}_{*})<\hat{L}_{\beta}(\hat{w}_{k})<\hat{L}_{\beta}(\hat{w}_{*})+\eta.
\]
By the uniformized KL property, we have
\begin{equation}\label{YKL}
\varphi ' (\hat{L}_{\beta}(\hat{w}_{k})-\hat{L}_{\beta}(\hat{w}_{*}))\|\nabla\hat{L}_{\beta}(\hat{w}_{k})\|\geq 1,~\forall~k>\tilde{k}.
\end{equation}
Using the concavity of $\varphi$, we obtain
\[\varphi(\hat{L}_{\beta}(\hat{w}_{k})-\hat{L}_{\beta}(\hat{w}_{*}))-\varphi(\hat{L}_{\beta}(\hat{w}_{k+1})-\hat{L}_{\beta}(\hat{w}_{*}))
\geq  \varphi'(\hat{L}_{\beta}(\hat{w}_{k})-\hat{L}_{\beta}(\hat{w}_{*}))(\hat{L}_{\beta}(\hat{w}_{k})-\hat{L}_{\beta}(\hat{w}_{k+1})).\]
Combining this with $\varphi'(\hat{L}_{\beta}(\hat{w}_{k})-\hat{L}_{\beta}(\hat{w}_{*}))>0$ and Lemma \ref{lem5}, we derive
\[\begin{aligned}
\hat{L}_{\beta}(\hat{w}_{k})-\hat{L}_{\beta}(\hat{w}_{k+1}) &\leq \frac{\varphi(\hat{L}_{\beta}(\hat{w}_{k})-\hat{L}_{\beta}(\hat{w}_{*}))-\varphi(\hat{L}_{\beta}(\hat{w}_{k+1})
	-\hat{L}_{\beta}(\hat{w}_{*}))}{\varphi'(\hat{L}_{\beta}(\hat{w}_{k})-\hat{L}_{\beta}(\hat{w}_{*}))}\\
&\overset{\eqref{YKL}}{\leq} \Xi_{k,k+1} \|\nabla \hat{L}_{\beta}(\hat{w}_{k})\| \leq c \Gamma_k \Xi_{k,k+1},
\end{aligned}\]
where $\Xi_{p,q+1}=\varphi(\hat{L}_{\beta}(\hat{w}_{p})-\hat{L}_{\beta}(\hat{w}_{*}))
-\varphi(\hat{L}_{\beta}(\hat{w}_{q+1})-\hat{L}_{\beta}(\hat{w}_{*}))$ and $\Gamma_k=\|{x}_{k}-x_{k-1}\|+\|{y}_{k}-y_{k-1}\|+\|{y}_{k-1}-y_{k-2}\|$.
From \eqref{xleqd} and Lemma \ref{lem3}, we have
\[\delta^x\|{x}_{k+1}-x_{k}\|^2+\delta^y \|{y}_{k+1}-y_{k}\|^2 \leq \delta^x \|d_{k}^x\|^2+\delta^y \|d_{k}^y\|^2 \leq c \Gamma_k\Xi_{k,k+1},~\forall~k>\tilde{k}.\]
Let $\xi=\min\{\delta^x,\delta^y\}>0$. Then
\[\|{x}_{k+1}-x_{k}\|^2+\|{y}_{k+1}-y_{k}\|^2 \leq \frac{c}{\xi} \Gamma_k \Xi_{k,k+1},~\forall~k>\tilde{k}.\]
This implies
\[(\|{x}_{k+1}-x_{k}\|+\|{y}_{k+1}-y_{k}\|)^2 \leq 2(\|{x}_{k+1}-x_{k}\|^2+\|{y}_{k+1}-y_{k}\|^2) \leq 2\frac{c}{\xi} \Gamma_k \Xi_{k,k+1}.\]
Thus, for all $i>\tilde{k}$,
\[
4(\|{x}_{i+1}-x_{i}\|+\|{y}_{i+1}-y_{i}\|)\leq 2\sqrt{(\frac{8c}{\xi}\Xi_{i,i+1})\Gamma_i}\leq \Gamma_i+\frac{8c}{\xi}\Xi_{i,i+1}.
\]
Summing up the inequality above for $i=k+1, \dots, q$, we obtain
\[
4\sum_{i={k}+1}^{q} (\|{x}_{i+1}-x_{i}\|+\|{y}_{i+1}-y_{i}\|)\leq \sum_{i={k}+1}^{q} \Gamma_i+\frac{8c}{\xi}\Xi_{{k}+1,q+1}.
\]
Since $\varphi(\hat{L}_{\beta}(\hat{w}_{q+1})-\hat{L}_{\beta}(\hat{w}_{*}))\geq 0$, it follows that
\begin{align}
&3\sum_{i={k}+1}^{q}\|{x}_{i+1}-x_{i}\|+2\sum_{i={k}+1}^{q}\|{y}_{i+1}-y_{i}\|\nonumber\\
\leq & \|{x}_{{k}+1}-x_{{k}}\|+2\|{y}_{{k}+1}-y_{{k}}\|+\|{y}_{{k}}-y_{{k}-1}\|-\|{x}_{q+1}-x_{q}\|-2\|{y}_{q+1}-y_{q}\|\nonumber\\
& -\|{y}_{q}-y_{q-1}\|+\frac{8c}{\xi}\varphi(\hat{L}_{\beta}(\hat{w}_{{k}+1})-\hat{L}_{\beta}(\hat{w}_{*}))\nonumber\\
\leq& \|{x}_{{k}+1}-x_{{k}}\|+2\|{y}_{{k}+1}-y_{{k}}\|+\|{y}_{{k}}-y_{{k}-1}\| +\frac{8c}{\xi}\varphi(\hat{L}_{\beta}(\hat{w}_{{k}+1})-\hat{L}_{\beta}(\hat{w}_{*})).\label{sumxy}
\end{align}
Taking $k=\tilde{k}$ and letting $q\to\infty$, we conclude
\[3\sum_{i={k}+1}^{+\infty}\|{x}_{i+1}-x_{i}\|+2\sum_{i={k}+1}^{+\infty}\|{y}_{i+1}-y_{i}\|<+\infty.\]
Thus, $\sum_{k=0}^{+\infty} \|{x}_{k+1}-x_{k}\|<+\infty$ and $\sum_{k=0}^{+\infty} \|{y}_{k+1}-y_{k}\|<+\infty$. From \eqref{lk1lk} and \eqref{tyk1yk}, we also have $\sum_{k=0}^{+\infty} \|{\lambda}_{k+1}-\lambda_{k}\|<+\infty$. Therefore, $\sum_{k=0}^{+\infty} \|{w}_{k+1}-w_{k}\| <+\infty$.
This implies that $\{w_k\}$ is a Cauchy sequence, and hence convergent. Combining this with Theorem \ref{the1} (iv,v), the proof of Theorem \ref{the2} is completed.
\end{proof}

Finally, we discuss the convergence rate of the HAP-PRS-SQP algorithm.
\begin{theorem}\label{the3}
Suppose that Assumptions \ref{ass1} and \ref{ass3} hold. If $\hat{L}_{\beta}(\hat{w})$ is a KL function with the correlation function $\varphi(t)=at^{1-\vartheta}$, where $\vartheta\in [0,1)$, $a>0$, then the following three conclusions hold.

(i) If $\vartheta=0$, the HAP-PRS-SQP algorithm terminates at a stationary point of $L_{\beta}(w)$  after a finite number of iterations;

(ii) If $\vartheta\in (0, \frac{1}{2}]$, there exists a constant $\tau \in (0,1)$ such that $\|w_k-w_*\|=\mathcal{O}(\tau ^{k})$;

(iii) If $\vartheta\in (\frac{1}{2},1)$, then $\|w_k-w_*\|=\mathcal{O}(k^{\frac{1-\vartheta}{1-2\vartheta}})$.
\end{theorem}
\begin{proof}
From Theorem \ref{the2}, the sequence $\{w_k\}$ generated by HAP-PRS-SQP converges to a stationary point $w_*$ of $L_{\beta}(w)$. We proceed by analyzing the cases $\vartheta = 0$ and $\vartheta > 0$ separately.

Case (i) ($\vartheta = 0$):  
If $\vartheta=0$, then $\varphi(t)=at$, $\varphi'(t)\equiv a$. If there exists an iterate index $k_{0}$ such that $\hat{L}_\beta(\hat{w}_{k_{0}})= \hat{L}_\beta(\hat{w}_{*})$, then HAP-PRS-SQP terminates at the $k_0$-th iterate, as shown in part (i) of the proof of Theorem \ref{the2}. If $\hat{L}\beta(\hat{w}k) > \hat{L}\beta(\hat{w})$ for all $k$, then by part (ii) of the proof of Theorem \ref{the2} and the KL property, we have 
\[\varphi ' (\hat{L}_{\beta}(\hat{w}_{k})-\hat{L}_{\beta}(\hat{w}_{*}))\|\nabla\hat{L}_{\beta}(\hat{w}_{k})\|=a\|\nabla\hat{L}_{\beta}(\hat{w}_{k})\|\geq 1.\]
This contradicts Lemma \ref{lem5} and the convergence $w_k \to w*$. Thus, claim (i) holds.

Case (ii) ($\vartheta > 0$): 
Since \eqref{tyk1yk}, we have
\[
\|d_{k-1}^y\|\leq \mathcal{O}(\|y_{k}-y_{k-1}\|)\leq\mathcal{O}(\|y_{k}-y_{*}\|)+\mathcal{O}(\|y_{k-1}-y_{*}\|).
\]
From \eqref{lambdak1}, $\lambda_*=\nabla g(y_*)$, $Ax_*=y_*$, \eqref{NF1}, \eqref{NF5} and \eqref{sd-y}, we derive
\begin{align}
\|\lambda_{k}-\lambda_{*}\|= &\|-\nabla g (y_{k-1})+\nabla g (y_{*})+(1-s)\beta [A( x_{k}-x_{*})-(y_{k}-y_{*})]\nonumber\\
&+\beta(y_{k}-y_*)-\beta(y_{k-1}-y_{*})-\frac{1}{1+\alpha}\mathcal{H}_{k-1}^y d_{k-1}^y\|\nonumber\\ 
\leq &\mathcal{O}(\|{x}_{k}-x_{*}\|)+\mathcal{O}(\|{y}_{k}-y_{*}\|)+\mathcal{O}(\|{y}_{k-1}-y_{*}\|)+\mathcal{O}(\|d_{k-1}^y\|)\nonumber\\
\leq & \mathcal{O}(\|{x}_{k}-x_{*}\|)+\mathcal{O}(\|{y}_{k}-y_{*}\|)+\mathcal{O}(\|{y}_{k-1}-y_{*}\|). \label{*4.3N}
\end{align}
Define $S_{k}:=\sum_{i=k}^{+\infty} (\|{x}_{i+1}-x_{i}\|+\|{y}_{i+1}-y_{i}\|)$. By Theorem \ref{the2}, it has $S_{k}\to 0$. Using \eqref{sumxy}, $\varphi(t)=at^{1-\vartheta}$ and $S_{k-1}-S_{k+1}=\|{x}_{k}-x_{k-1}\|+\|{y}_{k}-y_{k-1}\|+\|{x}_{k+1}-x_{k}\|+\|{y}_{k+1}-y_{k}\|$, we obtain
\begin{align}
S_{{k}+1}\leq& \frac{3}{2}\sum_{i=k+1}^{+\infty}\|{x}_{i+1}-x_{i}\|+\sum_{i=k+1}^{+\infty}\|{y}_{i+1}-y_{i}\|\nonumber\\
\leq& \frac{1}{2}\|{x}_{k+1}-x_{k}\|+\|{y}_{k+1}-y_{k}\|+\frac{1}{2}\|{y}_{k}-y_{k-1}\|+\frac{4c}{\xi}\varphi(\hat{L}_{\beta}(\hat{w}_{k+1})-\hat{L}_{\beta}(\hat{w}_{*}))\nonumber\\
\leq& \|{x}_{k+1}-x_{k}\|+\|{y}_{k+1}-y_{k}\|+\|{y}_{k}-y_{k-1}\|+\frac{4c}{\xi}\varphi(\hat{L}_{\beta}(\hat{w}_{k+1})-\hat{L}_{\beta}(\hat{w}_{*}))\nonumber\\
\leq& (S_{k-1}-S_{k+1}) +\frac{4c}{\xi}a(\hat{L}_{\beta}(\hat{w}_{{k}+1})-\hat{L}_{\beta}(\hat{w}_{*}))^{1-\vartheta},~\forall~k\geq \tilde{k}.\label{4.200N}
\end{align}
Since $\hat{L}_{\beta}$ satisfies the KL property at $w_{*}$ and $\varphi'(t)=a(1-\vartheta)t^{-\vartheta}$, we have 
\[
\varphi^{'}(\hat{L}_{\beta}(\hat{w}_{{k}+1})-\hat{L}_{\beta}(\hat{w}_*)) \|\nabla\hat{L}_{\beta}(\hat{w}_{k})\|
= a(1-\vartheta)(\hat{L}_{\beta}(\hat{w}_{{k}+1})-\hat{L}_{\beta}(\hat{w}_*))^{-\vartheta}\|\nabla\hat{L}_{\beta}(\hat{w}_{k})\| \geq 1.
\]
Combining this with Lemma \ref{lem5} yields
\[(\hat{L}_{\beta}(\hat{w}_{{k}+1})-\hat{L}_{\beta}(\hat{w}_*))^{\vartheta} \leq a(1-\vartheta)\|\nabla\hat{L}_{\beta}(\hat{w}_{k})\|\leq a(1-\vartheta)c(S_{k-1}-S_{k+1}).\]
From the inequality above and \eqref{4.200N}, there exists $\varepsilon>0$ such that
\begin{equation}\label{4.230N}
S_{{k}+1} \leq (S_{k-1}-S_{k+1}) +\varepsilon (S_{k-1}-S_{k+1})^{\frac{1-\vartheta}{\vartheta}},~\forall~k\geq \tilde{k}.
\end{equation}
Additionally, it follows that
\[\begin{aligned}
S_{k+1, q}:=&\sum_{i=k+1}^{q}(\|{x}_{i+1}-x_{i}\|+\|{y}_{i+1}-y_{i}\|) \\
\geq& \sum_{i=k+1}^{q}[(\|{x}_{i}-x_{*}\|-\|{x}_{i+1}-x_{*}\|)+(\|{y}_{i}-y_{*}\|-\|{y}_{i+1}-y_{*}\|)]\\
=&(\|{x}_{k+1}-x_{*}\|-\|{x}_{q+1}-x_{*}\|)+(\|{y}_{k+1}-y_{*}\|-\|{y}_{q+1}-y_{*}\|).
\end{aligned}\]
This means that
\begin{align}
S_{k+1}&= \lim_{q\to\infty} S_{k+1, q} \nonumber\\
&\geq  (\|{x}_{k+1}-x_{*}\|-\lim_{q\to\infty}\|{x}_{q+1}-x_{*}\|)+(\|{y}_{k+1}-y_{*}\|-\lim_{q\to\infty}\|{y}_{q+1}-y_{*}\|) \nonumber\\
&=\|{x}_{k+1}-x_{*}\|+\|{y}_{k+1}-y_{*}\|.\label{*4.0N}
\end{align}
Based on \eqref{*4.3N}, \eqref{4.230N} and \eqref{*4.0N}, the remaining proof of claims (ii) and (iii) follows similarly to the analyses (B1) and (B2) in the proof of Theorem 4.3 in \cite{JLY21}, and is omitted here.
\end{proof}

	\begin{remark}\label{remark4}
		With inequality \eqref{ALDD-merit-dec} established, the convergence analysis for the ALDD case follows the same unified proof, except that the simpler merit function $L_{\beta}(w)$ replaces $\hat{L}_{\beta}(\hat{w})$. Specifically, \eqref{ALDD-merit-dec} together with the boundedness of the sequence $\{w_k\}$ implies that the sequence $\{L_{\beta}(w_k)\}$ is non-increasing and bounded below, hence convergent. Summing \eqref{ALDD-merit-dec} over $k$ yields
		\[
		\sum_{k=0}^\infty(\|d_{k}^x\|^2+\|d_{k}^y\|^2)<+\infty.
		\]
		The proof of the subsequential convergence is then identical to that of Theorem \ref{the1}. Using the identities
		\[
		\nabla_x \hat{L}_{\beta}(\hat{w}_{k+1})=\nabla_x L_{\beta}(w_{k+1}),~\nabla_y \hat{L}_{\beta}(\hat{w}_{k+1})=\nabla_y L_{\beta}(w_{k+1}),~\nabla_\lambda \hat{L}_{\beta}(\hat{w}_{k+1})=\nabla_\lambda L_{\beta}(w_{k+1}),
		\]
		and equations \eqref{pxl}, \eqref{NF45}, \eqref{NF47} and \eqref{NF55}, there exists a constant $c_1 > 0$ such that
		\[
		\|\nabla_{w} \hat{L}_{\beta}(w_{k+1})\| \leq c_1 \bigl(\|x_{k+1}-x_k\| + \|y_{k+1}-y_k\| + \|y_k - y_{k-1}\|\bigr).
		\]
		Consequently, the global convergence established in Theorems~\ref{the2} and \ref{the3} remains valid.
	\end{remark}

\section{Numerical experiments}\label{sec5}
In this section, we evaluate the numerical performance of the proposed HAP-PRS-SQP algorithm. Subsection \ref{subsec:lasso} first compares HAP-PRS-SQP with the classical SQP on a smooth LASSO problem. Subsection \ref{subsec:class} then investigates a regularized binary classification problem, where HAP-PRS-SQP is compared with two existing splitting SQP methods: Algorithm 2.1 in \cite{JZY20} (denoted MS-SQO) and Algorithm 1 in \cite{JZY22} (denoted PRS-SQP-DSM). All tests are performed using MATLAB R2016b on a 64-bit laptop with Intel i9-13900HX CPU and 32.0 GB RAM. 

\subsection{Smooth LASSO}\label{subsec:lasso}
The objective function of the least absolute shrinkage and selection operator (LASSO) problem is given by
\[
	\min _{x\in \mathbb{R}^n}\frac{1}{2}\|Ax-d\|^2+\tau\|x\|_1,
\]
with $A\in \mathbb{R}^{m \times n}~(m\ll n)$, $d \in \mathbb{R}^m$, $\|x\|_1=\sum_{i=1}^n\|x_i\|$ and $\tau>0$. 
To obtain a smooth approximation, we replace the $\ell_1$ norm with the Huber function \cite{Huber1964}, 
$$
\ell_\mu(x_i)=
\begin{cases}
	\frac{1}{2 \mu} x_i^2, & |x_i|<\mu; \\
	|x_i|-\frac{\mu}{2}, & \text {otherwise},
\end{cases}
$$
and define $h_\mu(x)=\sum_{i=1}^n \ell_\mu(x_i)$. This leads to the smooth LASSO problem
\begin{equation}\label{HELASSO}
	\min_x~F(x):= \frac{1}{2}\|Ax-d\|^2+\tau h_\mu(x).
\end{equation}
In this experiment, we set $f(x)=\tau h_\mu(x)$ and $g(Ax)=\frac{1}{2}\|Ax-d\|^2$. Since $\nabla f(x)=\tau \nabla h_\mu(x)$ with
\[
(\nabla h_\mu(x))_i=
\begin{cases}
	\frac{x_i}{\mu}, & \|x_i\| \leq \mu,\\
	\text{sign}(x_i), & \text {otherwise},
\end{cases}
\]
and the Huber function is not twice continuously differentiable, we use the method III in \cite[Section 5]{JLY21} to evaluate the matrix. $\nabla f(x)$ is Lipschitz continuous with constant $\frac{\tau}{\mu}$ and the Lipschitz continuity of $\nabla g(Ax)$ is evident. The data are generated by 
\[A = \text{randn}(m, n),~u = \text{sprandn}(n, 1, r),~d = A * u.\]

	We compare HAP-PRS-SQP with the classical SQP that solves a full QP subproblem at each iteration. The parameters for HAP-PRS-SQP are set to $\rho=0.1$, $\nu=0.9$, $\beta=1$, $\ell=\sigma=10$, $\alpha=5$, $r=0.1$ and $s=0.5$. Both algorithms terminate when the relative change in the variables falls below $10^{-4}$. 
	Table \ref{t4.2} reports the number of iterations (Iter), CPU time in seconds (Tcpu) and the objective function value at termination (OFV). Since our algorithm is based on an augmented Lagrangian framework, in which the constraint is penalized and the primal iterates are not strictly feasible until convergence, it is essential to report the feasibility residual $\|x^*-y^*\|_\infty$ (Fea) to demonstrate the algorithm's progress toward convergence.  
	The results in Table \ref{t4.2} show that both methods achieve comparable objective values, and that HAP-PRS-SQP takes half the computation time of the classical SQP. Although the classical SQP converges in fewer iterations, it incurs a much higher cost per iteration because it solves a large, dense KKT system at each step. As the problem dimensions increase, this cost becomes prohibitive. In contrast, HAP-PRS-SQP solves two smaller subproblems for $x$ and $y$ using the precomputed matrix $A^\top A$. Consequently, although HAP-PRS-SQP requires more iterations, it reduces the total computation time for large-scale instances by half. This advantage highlights the practical scalability of the splitting structure. 
	In terms of feasibility, HAP-PRS-SQP exhibits a moderate residual due to its gradual dual-update mechanism, and the classical SQP maintains high feasibility by directly enforcing constraints through the KKT system.

	\vspace{-0.5cm}
	\begin{table*}[h]
		\centering
		\caption{Numerical comparisons for solving Problem \eqref{HELASSO} with $\tau=10^{-3}$ and $\mu=0.1$.}\label{t4.2}
			\resizebox{1\textwidth}{!}{
				\begin{tabular*}{\textheight}{@{\extracolsep\fill}cccccccccc}
					\toprule
					\multicolumn{2}{c}{Setting} & \multicolumn{4}{c}{HAP-PRS-SQP} & \multicolumn{4}{c}{Classical SQP} \\
					\cmidrule(r){1-2}\cmidrule(r){3-6}\cmidrule(r){7-10}
					m & n & Iter & Tcpu & OFV & Fea & Iter & Tcpu & OFV & Fea  \\
					\midrule
					200 & 800  & 20 & 1.27 & 0.6005 & 8.21$\times 10^{-3}$ & 8 & 1.28 & 0.6012 & 7.11$\times 10^{-14}$ \\
					300 & 1000 & 21 & 2.08 & 0.7035 & 6.60$\times 10^{-3}$ & 8 & 2.39 & 0.7040 & 9.24$\times 10^{-14}$ \\
					512 & 1536 & 20 & 6.26 & 1.1140 & 1.46$\times 10^{-2}$ & 26 & 27.90 & 1.1127 & 3.90$\times 10^{-10}$ \\
					600 & 2000 & 20 & 11.14 & 1.4147 & 1.41$\times 10^{-2}$ & 6 & 15.49 & 1.4152 & 2.49$\times 10^{-13}$ \\
					768 & 2688 & 20 & 23.19 & 1.9853 & 1.72$\times 10^{-2}$ & 8 & 43.63 & 1.9866 & 4.12$\times 10^{-13}$ \\
					900 & 3000 & 20 & 31.58 & 2.1329 & 1.40$\times 10^{-2}$ & 8 & 69.43 & 2.1310 & 1.99$\times 10^{-13}$\\
					\midrule
					\multicolumn{2}{c}{Total} & 121 & 75.52 & 7.9509 & & 64 & 160.12 & 7.9507 & \\
					\botrule
				\end{tabular*}
		}
	\end{table*}
	\vspace{-0.5cm}

\subsection{Regularized binary classification problem}\label{subsec:class}
Given a set of data pairs $(a_i, b_i) \in \mathbb{R}^n \times\{-1,1\}$ for $i\in \{1,2,\dots,T\}$, the $L_2$-regularized binary classification problem \cite{JZY22,MXC19} aims to solve the following optimization problem:
\begin{equation}\label{L2reg}
	\min_{x \in \mathbb{R}^n}~\frac{1}{T} \sum_{i=1}^T[1-\tanh (b_i a_i^{\top} x)]+\frac{\mu}{2}\|Ax\|^2,
\end{equation}
where the sigmoid loss is nonconvex and smooth, and $\mu>0$ is a regularization parameter. In this experiment, we set $\mu  = 0.001$ and define the matrix as
\[
A=\left(\begin{array}{ccccc}
	-1 & 1 & & &\\
	& -1 & 1 & &\\
	& & \ddots & \ddots &\\
	& & & -1 & 1
\end{array}\right) \in \mathbb{R}^{(n-1) \times n}.
\]
The data matrix $D=(a_1,\dots,a_T)$ and label vector $d=(b_1,\dots,b_T)^\top$ are generated by
\[D=\text{normc}(\text{randn}(n,T)),~d=\text{randsrc}(T,1)\]
Problem \eqref{L2reg} can be solved with HAP-PRS-SQP by setting $f(x)=\frac{\mu}{2}\|Ax\|^2$ and $g(y)=\frac{1}{T} \sum_{i=1}^T[1-\tanh (b_i a_i^{\top} y)]$ subject to the linear constraint $x-y=0$. Both $\nabla f(x)$ and $\nabla g(y)$ are Lipschitz continuity, see \cite[Section 6.2]{YTJ24}. The Hessian approximations in HAP-PRS-SQP are taken as
\[
\begin{aligned}
	H^x_{k+1}&=\nabla^2 f(x_{k+1})=\mu A^TA,\\
	H^y_{k+1}&=\nabla^2 g(y_{k+1})=\frac{2}{T} \sum_{i=1}^T \tanh (b_i a_i^{\top} y_{k+1})\left[1-\tanh ^2(b_i a_i^{\top} x)\right] a_i a_i^{\top},
\end{aligned}
\]
leading to the regularized Hessians
\[
\mathcal{H}_{k+1}^x=\nabla^2 f(x_{k+1})+ (\beta+\ell) I_n,~\mathcal{H}_{k+1}^y=\nabla^2 g(y_{k+1})+(\beta+\sigma) I_n.
\]

We first investigate the influence of the dual stepsizes $(r,s)$ in HAP-PRS-SQP. The other parameters are set to $\rho=0.4$, $\nu=0.6$, $\beta=1$, $\ell=5$, $\sigma=10$ and $\alpha=0$. All runs start from the initial point $(x_0,y_0,\lambda_0)=(0,0,0)$ and stop when the relative change of the variables falls below $10^{-4}$. 
The results in Table \ref{t4.1-1} indicate that the ALDA variant ($r+s>0$) yields slightly better objective values and feasibility, while the ALDD variant ($r+s<0$) tends to be faster in terms of both iterations and CPU time. 
Next, we examine the effect of the acceleration parameter $\alpha$ in HAP-PRS-SQP. Fig. \ref{f4.1-1} demonstrates that a larger $\alpha$ leads to a better final objective value, whereas a smaller $\alpha$ promotes a faster reduction of the relative accuracy during the early iterations.

\begin{figure}[]
	\begin{subfigure}{0.5\linewidth}
		\centering
		\includegraphics[height=0.7\textwidth]{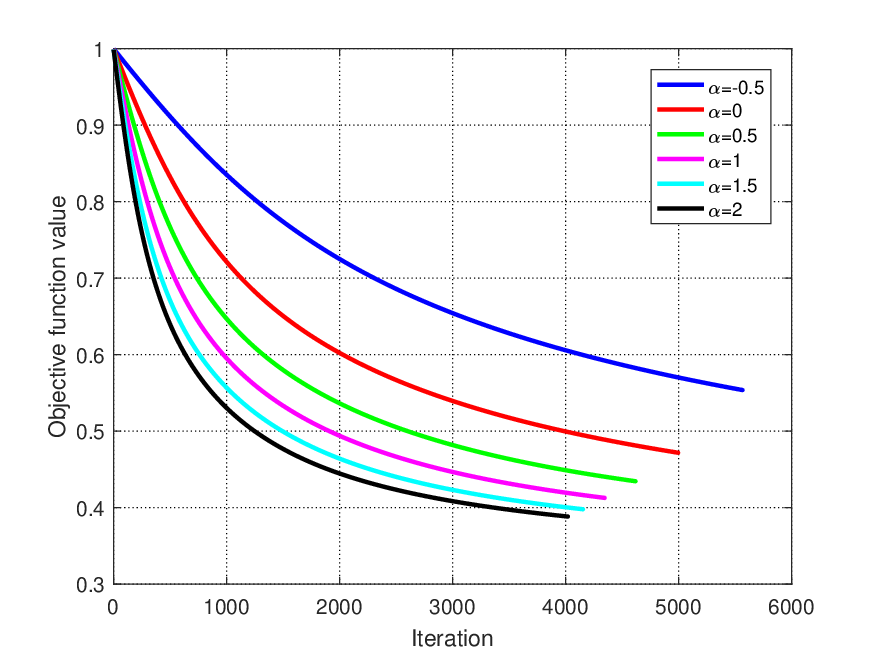}
	\end{subfigure}
	\begin{subfigure}{0.5\linewidth}
		\centering
		\includegraphics[height=0.7\textwidth]{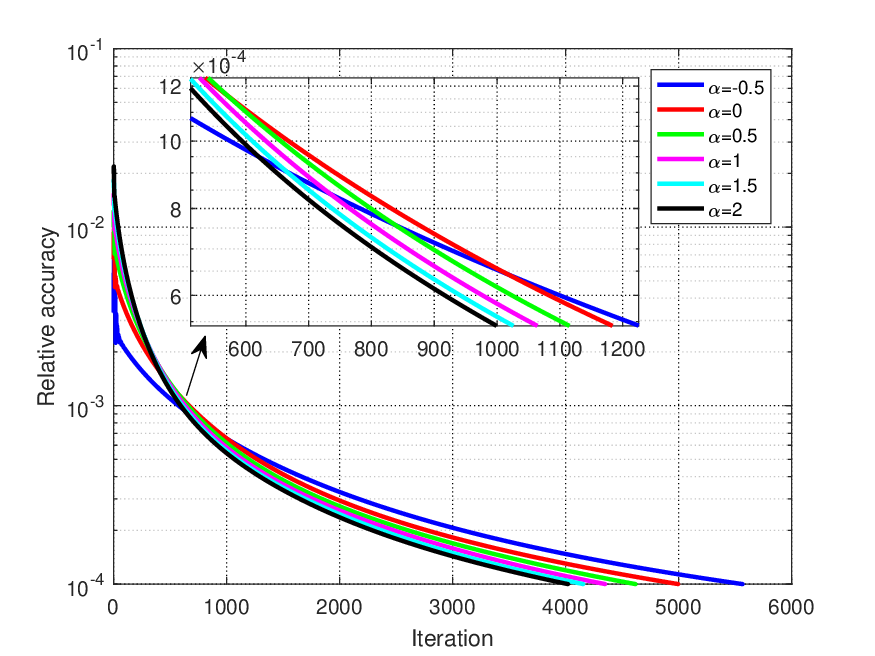}
	\end{subfigure}
	\caption{Performance profiles of HAP-PRS-SQP with different $\alpha$ on Problem \eqref{L2reg} with $n=T=100$.}\label{f4.1-1}
\end{figure}

\begin{table*}[]
	\belowrulesep=0pt
	\aboverulesep=0pt
	\centering
	\caption{Numerical results of HAP-PRS-SQP with different $(r,s)$ on solving \eqref{L2reg}.}\label{t4.1-1}
	\resizebox{1\textwidth}{!}{
		\begin{tabular}{ccccccccccc}
			\toprule
			\multicolumn{3}{c}{Dual update}  & \multicolumn{4}{c}{Problem \eqref{L2reg} with $n=T=100$} & \multicolumn{4}{c}{Problem \eqref{L2reg} with $n=T=1000$} \\
			\cmidrule(r){1-3}\cmidrule(r){4-7}\cmidrule(r){8-11}
			$r+s$& $s$  & $r$  & Iter & Tcpu & OFV & Fea & Iter   & Tcpu   & OFV   & Fea  \\
			\midrule
			\multicolumn{1}{c|}{\multirow{12}{*}{$>0$}} & \multicolumn{1}{c|}{\multirow{8}{*}{1}} & \multicolumn{1}{c|}{5} & 4973 & 6.33 & 0.4689 & \multicolumn{1}{c|}{1.99$\times 10^{-4}$} & 6695 & 299.19 & 0.8589 & 5.50$\times 10^{-5}$\\
			\multicolumn{1}{c|}{} & \multicolumn{1}{c|}{} & \multicolumn{1}{c|}{3}& 4976 & 6.20 & 0.4692 & \multicolumn{1}{c|}{1.79$\times 10^{-4}$} & 6698 & 300.26 & 0.8591 & 4.94$\times 10^{-5}$\\
			\multicolumn{1}{c|}{} & \multicolumn{1}{c|}{} & \multicolumn{1}{c|}{1}& 4984 & 6.24 & 0.4702 & \multicolumn{1}{c|}{1.19$\times 10^{-4}$} & 6707 & 297.99 & 0.8598 & 3.28$\times 10^{-5}$\\
			\multicolumn{1}{c|}{} & \multicolumn{1}{c|}{} & \multicolumn{1}{c|}{0.1}& 4998 & 6.24 & 0.4717 & \multicolumn{1}{c|}{2.16$\times 10^{-5}$} & 6722 & 298.09 & 0.8608 & 5.97$\times 10^{-6}$\\ \cmidrule{3-11}
			\multicolumn{1}{c|}{} & \multicolumn{1}{c|}{} & \multicolumn{1}{c|}{-0.1} & 5004 & 6.31 & 0.4725 & \multicolumn{1}{c|}{2.62$\times 10^{-5}$} & 6729 & 297.81 & 0.8613 & 7.13$\times 10^{-6}$\\
			\multicolumn{1}{c|}{} & \multicolumn{1}{c|}{} & \multicolumn{1}{c|}{-0.3}  & 5014 & 6.33 & 0.4737  & \multicolumn{1}{c|}{1.01$\times 10^{-4}$} & 6740 & 298.11 & 0.8621 & 2.75$\times 10^{-5}$\\
			\multicolumn{1}{c|}{} & \multicolumn{1}{c|}{} & \multicolumn{1}{c|}{-0.5} & 5031 & 6.23 & 0.4759  & \multicolumn{1}{c|}{2.33$\times 10^{-4}$} & 6759 & 298.72 & 0.8635 & 6.36$\times 10^{-5}$\\
			\multicolumn{1}{c|}{} & \multicolumn{1}{c|}{} & \multicolumn{1}{c|}{-0.7} & 5070 & 6.44 & 0.4808 & \multicolumn{1}{c|}{5.35$\times 10^{-4}$} & 6802 & 302.39 & 0.8667 & 1.45$\times 10^{-4}$\\ \cmidrule{2-11}
			\multicolumn{1}{c|}{\multirow{20}{*}{$<0$}} & \multicolumn{1}{c|}{\multirow{8}{*}{-1}} & \multicolumn{1}{c|}{5}& 5040 & 6.37 & 0.4768 & \multicolumn{1}{c|}{2.91$\times 10^{-4}$} & 6769 & 301.09 & 0.8641 & 7.93$\times 10^{-5}$\\
			\multicolumn{1}{c|}{} & \multicolumn{1}{c|}{} & \multicolumn{1}{c|}{4} & 5043 & 6.26 & 0.4771 & \multicolumn{1}{c|}{3.10$\times 10^{-4}$} & 6772 & 300.37 & 0.8643 & 8.44$\times 10^{-5}$\\
			\multicolumn{1}{c|}{} & \multicolumn{1}{c|}{} & \multicolumn{1}{c|}{3} & 5048 & 6.24 & 0.4777 & \multicolumn{1}{c|}{3.48$\times 10^{-4}$} & 6777 & 300.40 & 0.8647 & 9.47$\times 10^{-5}$\\
			\multicolumn{1}{c|}{} & \multicolumn{1}{c|}{} & \multicolumn{1}{c|}{2} & 5063 & 6.26 & 0.4795 & \multicolumn{1}{c|}{4.60$\times 10^{-4}$} & 6794 & 304.34 & 0.8659 & 1.25$\times 10^{-4}$\\ \cmidrule{1-1}\cmidrule{3-11}
			\multicolumn{1}{c|}{} & \multicolumn{1}{c|}{} & \multicolumn{1}{c|}{1-$10^{-2}$} & 7091 & 8.97 & 1.5711 & \multicolumn{1}{c|}{1.97$\times 10^{-2}$} & 9838 & 436.91 & 1.1559 & 5.95$\times 10^{-3}$\\
			\multicolumn{1}{c|}{} & \multicolumn{1}{c|}{} & \multicolumn{1}{c|}{1-$10^{-3}$}& 3463 & 4.36 & 1.0245 & \multicolumn{1}{c|}{2.55$\times 10^{-2}$}& 62 & 2.74 & 0.9992 & 3.14$\times 10^{-3}$\\
			\multicolumn{1}{c|}{} & \multicolumn{1}{c|}{} & \multicolumn{1}{c|}{1-$10^{-4}$}& 84 & 0.11 & 0.9904 & \multicolumn{1}{c|}{2.43$\times 10^{-2}$} & 62 & 2.72 & 0.9992 & 3.08$\times 10^{-3}$\\
			\multicolumn{1}{c|}{} & \multicolumn{1}{c|}{} & \multicolumn{1}{c|}{1-$10^{-5}$}& 85 & 0.14 & 0.9904  & \multicolumn{1}{c|}{2.43$\times 10^{-2}$}& 62 & 2.73 & 0.9992 & 3.07$\times 10^{-3}$\\ \cmidrule{2-11}
			\multicolumn{1}{c|}{} & \multicolumn{1}{c|}{\multirow{4}{*}{1}} & \multicolumn{1}{c|}{-1-$10^{-2}$} & 7183 & 9.01 & 1.5641 & \multicolumn{1}{c|}{1.96$\times 10^{-2}$} & 9847 & 436.12 & 1.1498 & 5.84$\times 10^{-3}$\\
			\multicolumn{1}{c|}{} & \multicolumn{1}{c|}{} & \multicolumn{1}{c|}{-1-$10^{-3}$}& 90 & 0.12 & 0.9908 & \multicolumn{1}{c|}{2.46$\times 10^{-2}$} & 74 & 3.25 & 0.9991 & 3.12$\times 10^{-3}$\\
			\multicolumn{1}{c|}{} & \multicolumn{1}{c|}{} & \multicolumn{1}{c|}{-1-$10^{-4}$}& 98 & 0.12 & 0.9904 & \multicolumn{1}{c|}{2.43$\times 10^{-2}$} & 78 & 3.42 & 0.9991 & 3.09$\times 10^{-3}$\\
			\multicolumn{1}{c|}{} & \multicolumn{1}{c|}{} & \multicolumn{1}{c|}{-1-$10^{-5}$}& 99 & 0.15 & 0.9904 & \multicolumn{1}{c|}{2.42$\times 10^{-2}$} & 78 & 3.49 & 0.9991 & 3.09$\times 10^{-3}$\\ \cmidrule{2-11}
			\multicolumn{1}{c|}{} & \multicolumn{1}{c|}{\multirow{4}{*}{-$10^{-4}$}} & \multicolumn{1}{c|}{-$10^{-4}$} & 4998 & 6.22 & 0.4710 & \multicolumn{1}{c|}{1.48$\times 10^{-2}$} & 6721 & 301.35 & 0.8604 & 5.10$\times 10^{-3}$\\
			\multicolumn{1}{c|}{} & \multicolumn{1}{c|}{} & \multicolumn{1}{c|}{-$10^{-5}$} & 5001 & 6.23 & 0.4713 & \multicolumn{1}{c|}{9.63$\times 10^{-3}$}& 6726 & 298.51 & 0.8606 & 3.15$\times 10^{-3}$\\
			\multicolumn{1}{c|}{} & \multicolumn{1}{c|}{}& \multicolumn{1}{c|}{-$10^{-6}$} & 5002 & 6.28 & 0.4713 & \multicolumn{1}{c|}{9.23$\times 10^{-3}$} & 6726 & 300.47 & 0.8606 & 3.00$\times 10^{-3}$\\
			\multicolumn{1}{c|}{} & \multicolumn{1}{c|}{} & \multicolumn{1}{c|}{-$10^{-7}$}& 5002 & 6.32 & 0.4713 & \multicolumn{1}{c|}{9.19$\times 10^{-3}$} & 6726 & 295.43 & 0.8606 & 2.99$\times 10^{-3}$\\
			\botrule
		\end{tabular}
	}
\end{table*}

	We then compare HAP-PRS-SQP with the classical SQP method. The parameters for HAP-PRS-SQP are set to $\rho=0.1$, $\nu=0.5$, $\beta=1$, $\ell=5$, $\sigma=10$, $\alpha=8$, $r=230$ and $s=10$. Both algorithms are permitted a maximum of 30 iterations and terminate earlier if the relative change of the variables falls below $10^{-3}$. 
	Table \ref{t4.1} shows that HAP-PRS-SQP consistently achieves lower objective values in significantly less CPU time. Since both methods complete the same iterations, the efficiency of HAP-PRS-SQP stems from updating the variables using lower-cost subproblems at each iteration.  
	Fig. \ref{f4.1-2} further illustrates that while the classical SQP can easily reach feasible solutions due to its linear constraint enforcement, the objective value initially decreases rapidly but then plateaus. This behavior likely arises from the nonconvexity of the problem: the QP subproblem may become ill-conditioned or yield steps that do not provide sufficient descent. In contrast, HAP-PRS-SQP prioritizes rapid objective reduction and progressively improves feasibility through iterative dual updates.

\begin{table*}[]
	\centering
	\caption{Numerical results of HAP-PRS-SQP and classical SQP for solving Problem \eqref{L2reg} with $\mu=0.01$.}\label{t4.1}
		\resizebox{1\textwidth}{!}{
			\begin{tabular*}{\textheight}{@{\extracolsep\fill}cccccccc}
				\toprule
				\multicolumn{2}{c}{Setting} & \multicolumn{3}{c}{HAP-PRS-SQP} & \multicolumn{3}{c}{Classical SQP} \\
				\cmidrule(r){1-2}\cmidrule(r){3-5}\cmidrule(r){6-8}
				n & T  & Tcpu & OFV & Fea  & Tcpu & OFV & Fea  \\
				\midrule
				50 & 500 & 0.06 & 0.9221 & 9.00$\times 10^{-2}$  & 0.18 & 0.9274 & 1.80$\times 10^{-14}$ \\
				100 & 1000  & 0.21 & 0.9347 & 4.13$\times 10^{-2}$  & 0.51 & 0.9413 & 2.09$\times 10^{-14}$ \\
				500 & 5000  & 5.82 & 0.9867 & 9.62$\times 10^{-3}$  & 13.03 & 0.9926 & 9.77$\times 10^{-14}$ \\
				1000 & 7500 & 20.63 & 0.9898 & 6.97$\times 10^{-3}$  & 64.29 & 0.9929 & 1.30$\times 10^{-13}$ \\
				\midrule
				\multicolumn{2}{c}{Total} & 26.72 & 3.8333 &  & 78.01 & 3.8542 &\\
				\botrule
			\end{tabular*}
	}
\end{table*}

\begin{figure}[]
	\begin{subfigure}{0.5\linewidth}
		\centering
		\includegraphics[height=0.7\textwidth]{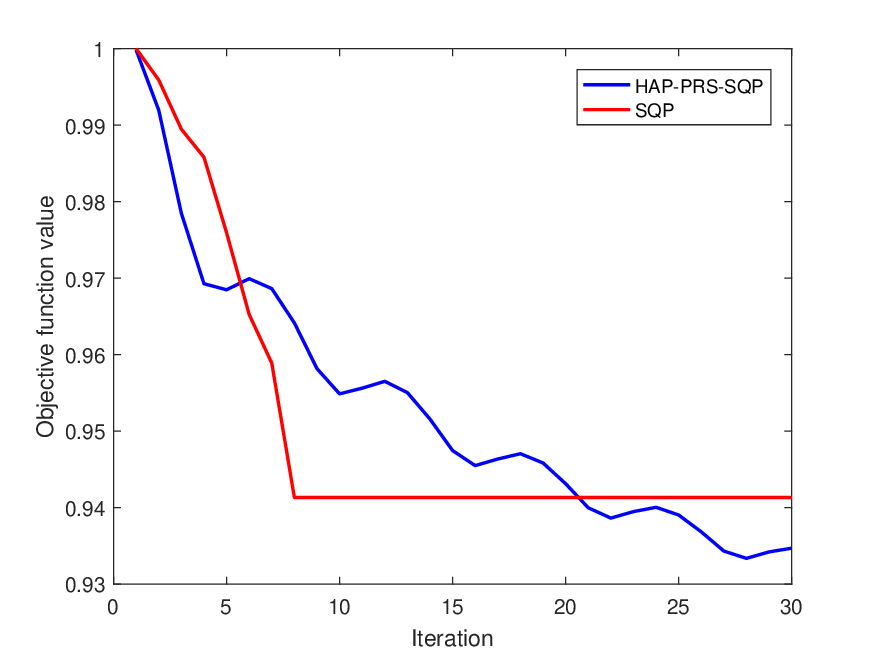}
	\end{subfigure}
	\begin{subfigure}{0.5\linewidth}
		\centering
		\includegraphics[height=0.7\textwidth]{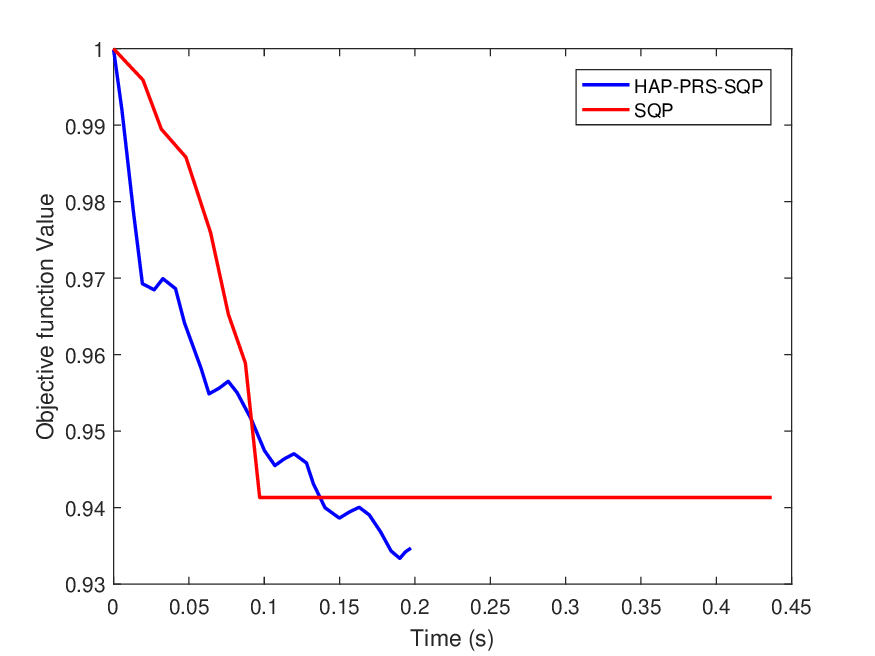}
	\end{subfigure}
	\caption{Performance profiles of HAP-PRS-SQP and classical SQP on Problem \eqref{L2reg} with $n=100$ and $T=1000$.}\label{f4.1-2}
\end{figure}

Finally, we compare HAP-PRS-SQP with the two existing splitting SQP methods, MS-SQO and PRS-SQP-DSM. The parameters for these methods are derived from the original papers. For HAP-PRS-SQP, we set $\rho=0.4$, $\nu=0.6$, $\beta=1$, $\ell=5$, $\sigma=10$, $\alpha=2$, $r=0.1$ and $s=1$. The results are reported in Table~\ref{t4.1-2}, where ``---'' indicates that the corresponding method failed to converge. Clearly, HAP-PRS-SQP outperforms the other two methods in terms of efficiency. It requires fewer iterations and less CPU time while achieving a much smaller feasibility residual and delivering better objective values. This superior performance can be attributed to two factors. First, the proximal terms and the acceleration step navigate the nonconvex landscape more effectively, enhancing numerical stability. Second, the two dual-update steps can be interpreted as an accelerated gradient ascent on the dual function. This maintains accurate estimates of the Lagrange multipliers and progressively reduces constraint violations.

\begin{table*}[h]
	\centering
	\caption{Numerical comparisons of the methods for solving Problem \eqref{L2reg}.}\label{t4.1-2}
	\resizebox{1\textwidth}{!}{
		\begin{tabular*}{\textheight}{@{\extracolsep\fill}ccccccccccccc}
			\toprule
			Setting & \multicolumn{4}{c}{PRS-SQP-DSM} & \multicolumn{4}{c}{MS-SQO}& \multicolumn{4}{c}{HAP-PRS-SQP} \\
			\cmidrule(r){1-1}\cmidrule(r){2-5}\cmidrule(r){6-9}\cmidrule(r){10-13}
			n=T & Iter & Tcpu & OFV & Fea & Iter & Tcpu & OFV & Fea & Iter & Tcpu & OFV & Fea \\
			\midrule
			100 & 2461 & 7.95 & 0.3229 & 1.68$\times 10^{-1}$&	4389 & 13.82 & 0.4091 & 2.40$\times 10^{-3}$&4002 & 4.76 & 0.3871 & 3.50$\times 10^{-5}$\\
			200 &2456 & 30.93 & 0.4316 & 1.39$\times 10^{-1}$&4252 & 53.39 & 0.5032 & 2.02$\times 10^{-3}$&3771 & 7.64 & 0.4851 & 2.17$\times 10^{-5}$\\
			300 & --- & --- & --- & --- & 4728 & 122.17 & 0.6141 & 1.68$\times 10^{-3}$&4324 & 15.74 & 0.5930 & 2.02$\times 10^{-5}$\\
			500 & --- & --- & --- & --- &5398 & 372.18 & 0.7003 & 1.29$\times 10^{-3}$&5024 & 61.76 & 0.6747 & 1.74$\times 10^{-5}$\\
			1000 & --- & --- & --- & --- & 6020 & 1661.73 & 0.8125 & 9.52$\times 10^{-4}$&5633 & 258.74 & 0.7898 & 1.05$\times 10^{-5}$\\
			\botrule
		\end{tabular*}
	}
\end{table*}

	Overall, the numerical results demonstrate that HAP-PRS-SQP offers a compelling balance of efficiency and scalability. On the smooth LASSO problem, it leverages its splitting structure to drastically reduce computational time relative to the classical SQP. On the binary classification problem, it achieves superior objective values and feasibility compared to existing splitting SQP methods. These advantages make HAP-PRS-SQP a practical and scalable choice for large-scale optimization.

\section{Conclusions}\label{sec6}
We have developed a proximal PRS-SQP algorithm with a hybrid acceleration technique, named HAP-PRS-SQP, for solving smooth composite optimization problems. The most original aspect of our study is the flexibility of the dual update, which can be either ascent or descent, significantly broadening its applicability. Our results show that the direction of dual updates does not affect the algorithm's convergence. The hybrid acceleration scheme, though simple, unifies two established strategies: inertial extrapolation and back substitution. Under the assumption that the objective function satisfies the KL property and the parameters satisfy  appropriate conditions, the sequence generated by HAP-PRS-SQP converges to a stationary point. Preliminary numerical results on regularized binary classification problems and smooth LASSO not only illustrate the behavior of different dual updates, but also confirm the effectiveness of the hybrid acceleration steps.

%\bmhead{Acknowledgements}
%This work was supported by the Natural Science Foundation of China (12261008, 12171106), the Natural Science Foundation of Guangxi Province (2023GXNSFAA026158), the Research Foundation of Guangxi Minzu University (2021KJQD04) and the Guangxi Scholarship Fund of Guangxi Education Department(GED). The authors are grateful to the anonymous reviewers for their valuable comments and suggestions, which have helped improve the quality of the manuscript.

%\bmhead{Data availability}
%The code is available on the GitHub repository at \href{https://github.com/JasonJiachenJin/HAP-PRS-SQP}{https://github.com/JasonJiachenJin/HAP-PRS-SQP}.

%\bmhead{Conflict of interest}
%The authors declare no conflict of interest.

\bibliography{sn-bibliography}% common bib file

\end{document}